\documentclass[10 pt, leqno]{amsart}

\usepackage {amsfonts}
\usepackage{amsthm}
\usepackage{amssymb}
\usepackage{latexsym}
\usepackage{amsmath}
\usepackage{mathrsfs}
\usepackage[all]{xy}
\usepackage{comment}
\usepackage{hyperref}

\theoremstyle{plain}
\newtheorem{tw}{Theorem}[section]

\newtheorem {lem} [tw]{Lemma}
\newtheorem {prop}[tw] {Proposition}

\newtheorem{cor}[tw]{Corollary}

\theoremstyle{definition}
\newtheorem {deft}[tw] {Definition}
\newtheorem {rem} [tw]{Remark}

\newcommand{\bc} {\Bbb C}
\newcommand{\bn}{\Bbb N}
\newcommand{\br}{\Bbb R}

\newcommand{\bt}{\Bbb T}
\newcommand{\bz}{\Bbb Z}

\newcommand{\alg} {\mathsf{A}}
\newcommand {\Tr} {{\textup{Tr}}}
\newcommand {\card} {{\textup{card}}}
\newcommand {\QG} {{\mathbb{G}}}
\newcommand {\QH} {{\mathbb{H}}}
\newcommand {\QK} {{\mathbb{K}}}
\newcommand {\QL} {{\mathbb{L}}}
\newcommand {\id} {{\textup{id}}}

\newcommand{\tu}{\textup}

\newcommand{\Hil}{\mathsf{H}}

\newcommand{\Jnd}{\mathcal{J}}
\newcommand{\ib}{\bar{i}}

\newcommand{\cat}{\mathcal{C}}

\newcommand{\Hom}{\tu{Hom}}

\newcommand{\Com}{\Delta}
\newcommand{\sph}{\mathbb{S}}

\newcommand{\ot}{\otimes}

\newcommand{\wt}{\widetilde}

\numberwithin{equation}{section}


\begin{document}

\author{Teodor Banica}
\address{Department of Mathematics, Cergy-Pontoise University, 95000 Cergy-Pontoise, France} \email{teodor.banica@u-cergy.fr}

\author{Adam Skalski}
\address{Department of Mathematics and Statistics,  Lancaster University,
Lancaster, LA1 4YF, United Kingdom \newline \hspace*{0.3cm} Institute of Mathematics of the Polish Academy of Sciences,
ul.\'Sniadeckich 8, 00-956 Warszawa, Poland} \email{a.skalski@lancaster.ac.uk}

\begin{abstract} We introduce and study natural two-parameter families of quantum groups motivated on one hand by the liberations of
classical orthogonal groups and on the other by quantum isometry groups of the duals of the free groups. Specifically, for each
pair $(p,q)$ of non-negative integers we define and investigate quantum groups $O^+(p,q)$, $B^+(p,q)$, $S^+(p,q)$ and $H^+(p,q)$ corresponding to,
respectively, orthogonal groups, bistochastic groups, symmetric groups and  hyperoctahedral groups. In the first three cases the
new quantum groups turn out to be related to the (dual free products of) free quantum groups studied earlier.
For $H^+(p,q)$ the situation is different and we show that $H^+(p,0) \approx
\tu{QISO}(\widehat{\mathbb{F}_p})$, where the latter can be viewed as a liberation of the classical isometry group of the
$p$-dimensional torus.

\end{abstract}

\title{\bf Two-parameter families of quantum symmetry groups}

\maketitle

\section*{Introduction}

Compact quantum groups have entered mathematics in late 1980s (see \cite{woronowicz98}, \cite{VanDaele} and references therein).
Recent years have brought an increased interest in investigating quantum groups as quantum symmetry or isometry groups of
classical or quantum spaces (\cite{BanJFA}, \cite{Teosurvey}, \cite{Deb}, \cite{ours}, \cite{TeoDeb}). One particular approach to
constructing quantum symmetry groups is the so-called `liberation' of classical compact groups. This technique, developed by the
first named author and his collaborators, is  based on choosing a suitable collection of relations satisfied by the functions on
the group in question and then relaxing the commutativity assumptions (\cite{TeoRoland}). On the other hand in recent work of
Bhowmick and the second named author (motivated by \cite{Deb}) a quantum isometry group has been associated with the (dual of)
each finitely generated discrete group; in particular the quantum isometry groups of the duals of the free groups were computed.
The last result and the form of the obtained quantum isometry groups suggested considering a general framework in which  a
variation of universal quantum orthogonal groups of \cite{univ} is realised by replacing the usual selfadjointness of entries by
imposing specific relations between entries and their adjoints. Although it turns out that this new choice actually does not
introduce nontrivial modifications on the level of quantum orthogonal groups, the situation changes if we consider  quantum
symmetric groups (\cite{Wang}) or quantum hyperoctahedral groups (\cite{hyperoct}). In this paper we present a full study of
these deformations and cast them in the language of `liberations' studied in \cite{TeoRoland}.

As we are going to consider at the same time deformed adjoint relations and the usual selfadjointness conditions on some other
entries, we will throughout the paper work with two parameters, $p$ (representing the deformed relations) and $q$ (the standard
ones). These parameters are assumed to be non-negative integers, not
simultaneously equal to $0$. As the deformed relations involve pairs of coordinates, the quantum groups that we study will have `rank' (the dimension of the
fundamental unitary representation) equal to $2p+q$.    We will consider deformed quantum versions of orthogonal, bistochastic, symmetric and
hyperoctahedral groups. Main body of the results obtained in this paper can be summarized in the following table (the
categories of representations are described in terms of noncrossing, possibly `bulleted', partitions, with the details given in
Section \ref{coreps-partitions}; $i(p,q)$ is defined to be equal to $1$ if $pq=0$, and to $2$ if $pq>0$):

\vspace*{0.3cm}
\begin{tabular}{|c|c|c|c|c|}\hline
 quantum group $\QG$ &  $O^+(p,q)$ & $B^+(p,q)$ & $S^+(p,q)$ & $H^+(p,q)$ \\ \hline
algebra $C(\QG)$ & $A_o(2p+q)$ & $A_b(p,q)$ & $A_h(p) \star A_s(q)$ & $A_h(p,q)$ \\ \hline cat.\ of reps & $NC_2$ & $NC_{12}$ &$NC_{\bullet} \star NC$ & $NC_{\tu{even}}$ \\
\hline classical version & $O_{2p+q}$ & $O_{2p+q-i(p,q)}$ & $H_p \times S_q$ &  $(\bt^p\rtimes H_p )\times H_q$ \\
\hline
\end{tabular}

\vspace*{0.3cm} As explained earlier, if $p=0$ we obtain the `liberated objects' studied in \cite{TeoRoland}. On the other hand
when $q=0$ we obtain the following (in particular $H^+(p,0)$  is the quantum isometry group of the $C^*$-algebra of the free
group $\mathbb{F}_p$ discovered in \cite{groupdual} and in a way providing a starting point for the considerations in this work):

\vspace*{0.3cm}

\begin{tabular}{|c|c|c|c|c|}\hline
 quantum group $\QG$ &  $O^+(p,0)$ & $B^+(p,0)$ & $S^+(p,0)$ & $H^+(p,0)$ \\ \hline
quantum symmetry group of  & $\mathbb{S}^{2p}$ & $\mathbb{S}^{2p}_{-}$ & $[0,1]^p$ & $\widehat{\mathbb{F}_p}$
\\
\hline classical version & $O_{2p}$ & $O_{2p-1}$ & $H_p $ & $\bt^p \rtimes H_{p}$ \\
\hline classical symmetry group of  & $\mathcal{S}^{2p}$ & $\mathcal{S}^{2p}_{-}$ & $[0,1]^p$ &$\mathbb{T}^p$
\\   \hline
\end{tabular}
\vspace*{0.3cm}

Above $\mathcal{S}^{2p}$ and $\mathbb{S}^{2p}$ denote respectively the usual sphere in $\br^{2p}$ and the free sphere studied in
\cite{TeoDeb}, and $\mathcal{S}^{2p}_{-}$ and $\mathbb{S}^{2p}_{-}$ denote the respective spheres with one coordinate
fixed. With the isomorphisms established in this paper the first two columns can be deduced from results in \cite{TeoDeb} (and the
third is a consequence of \cite{hyperoct}).

The detailed plan of the paper is as follows: in Section 1 we quote basic definitions, establish some terminology related to
compact quantum groups and recall the quantum (free) symmetry groups corresponding to orthogonal, symmetric, bistochastic and
hyperoctahedral groups. In Section 2 the study of their two-parameter counterparts begins with the analysis of the orthogonal,
bistochastic and symmetric quantum groups denoted respectively $O^+(p,q)$, $B^+(p,q)$ and $S^+(p,q)$. It turns out that all of
them can be described in terms of (the free products of) the one-parameter versions. Section 3 contains a detailed analysis of
the two-parameter quantum hyperoctahedral group $H^+(p,q)$; in particular we show that $H^+(p,0)$ coincides with the quantum
isometry group of the dual of the free group discovered in \cite{groupdual}. Section 4 is devoted to establishing the description of the
categories of representations of our quantum groups in terms of non-crossing (marked) partitions. Further in Section 5 theorems proved in Sections 2-4 are used to analyse the relations between the quantum groups studied in the paper: we investigate the
generation results, intersection results and inclusions of the form $\QG(p,0) \hat{\star}\, \QG (0,q) \subset \QG(p,q)$. We also
show there a fact conjectured in \cite{groupdual}: the two-parameter quantum hyperoctahedral group $H^+(p,q)$ may be viewed as a
quantum extension of the quantum symmetric group $S^+_{2p+q}$. In Section 6 we describe the classical versions of the quantum groups we
study to show that each of these quantum groups has a natural description as a liberation of a classical symmetry group. Finally in the last
section we discuss in what sense the family of quantum groups described in this paper exhausts the natural two-parameter
construction presented in Sections 2-3 and in the process discover another two-parameter quantum group, $H^+_s(p,q)$ which turns
out to be isomorphic to the dual free product of $H^{+4}_p$ and $H^+_q$.


\section{Compact quantum groups - definitions and notation} \label{Secdefcom}

In this section we recall the definition of a compact quantum group due to Woronowicz and introduce quantum orthogonal,
symmetric, hyperoctahedral and bistochastic groups. The minimal tensor product of $C^*$-algebras will be denoted by $\ot$,
algebraic tensor products by $\odot$. For $n \in \bn$ we denote the algebra of $n$ by $n$ complex matrices by $M_n$.

\begin{deft}
Let $\alg$ be a unital $C^*$-algebra and $\Com:\alg \to \alg\ot \alg$ be a unital $^*$-homomorphism satisfying the
coassociativity condition:
\[ (\Com \ot \id_{\alg}) \Com = (\id_{\alg} \ot \Com) \Com.\]
If additionally $\overline{\Com(\alg) (1 \ot \alg)} = \overline{\Com(\alg) (\alg \ot 1 )}= \alg \ot \alg$ we say that $\alg$ is
the algebra of continuous functions on a compact quantum group $\QG$ and usually write $\alg = C(\QG)$. A unique dense Hopf
$^*$-subalgebra of $C(\QG)$, the algebra of coefficients of finite-dimensional unitary representations of $\QG$,  will be
denoted by $R(\QG)$.
\end{deft}

  If $\QG$ is a compact quantum group and $n \in \bn$ then a unitary matrix $U=(U_{ij})_{i,j=1}^n\in M_n(C(\QG))$ is
called a \emph{fundamental representation} of $\QG$ (or a \emph{fundamental corepresentation} of $C(\QG)$) if for each $i,j
=1,\ldots,n$
\[\Com(U_{ij}) = \sum_{k=1}^n U_{ik} \ot U_{kj}\]
and the entries of $U$ generate $C(\QG)$  as a $C^*$-algebra. If $\QG$ admits a fundamental representation, it is called a
\emph{compact matrix quantum group}. This will be the case for all quantum groups considered in this paper; in fact they will be
defined via their respective fundamental representations.

In general there is an ambiguity in passing from $R(\QG)$ to $C(\QG)$
related to the fact that not all compact quantum groups are \emph{coamenable}. As all quantum groups studied in this paper will
be defined by universal properties, we will assume that $C(\QG)$ is the \emph{universal} completion of $R(\QG)$ (\cite{coamen}).

We will later need a free product construction  introduced in \cite{Wangfree}. If $\QG_1$, $\QG_2$ are compact quantum groups,
then the $C^*$-algebraic free product $C(\QG_1) \star C(\QG_2)$ has a natural structure of the algebra of continuous functions on
a compact quantum group, to be denoted $\QG_1 \hat{\star} \,\QG_2$. In particular if $U_1 \in M_n (C(\QG_1))$ and $U_2 \in M_n
(C(\QG_2))$ are respective fundamental corepresentations, then $\begin{bmatrix} U_1 & 0\\0&U_2\end{bmatrix} \in M_n (C(\QG_1)
\star C(\QG_2))$ is the fundamental corepresentation of $C(\QG_1 \hat{\star} \, \QG_2)$. The construction is dual to the usual
free product of discrete groups: when the quantum groups in question are duals of classical discrete groups, $\QG_1 =
\widehat{\Gamma_1}$, $\QG_2 = \widehat{\Gamma_2}$, then $\QG_1 \hat{\star} \,\QG_2 \approx \widehat{\Gamma_1 \star \Gamma_2}$.

 Let $n \in \bn$; it will denote the dimension of the fundamental representation of the compact quantum groups defined below.
 The following definition comes from \cite{univ}; here we recast it in the language described above.

\begin{deft}\label{uniorth}
Let $F\in M_n$ be an invertible matrix such that $F\bar{F}=cI_n$ for some $c \in \bc$. Let  $A_o(F)$ denote the universal
$C^*$-algebra generated by the entries of a unitary $U \in M_n \ot A_o(F)$ such that
\begin{equation} \label{UFbar} U = (F \ot 1) \bar{U} (F^{-1} \ot 1)\end{equation}
(here and in what follows a bar over the matrix denotes a matrix obtained by an entrywise conjugation of entries). When $U \in M_n
\ot A_o(F)$ is interpreted as the fundamental unitary corepresentation, we can view $A_o(F)$ as the algebra of continuous functions on the
compact quantum group denoted by $O^+(F)$. In particular if $F = I_n$, we write simply $A_o(n)$ and $O^+_n$ instead of $A_o(I_n)$
and $O^+(I_n)$.
\end{deft}

Recall that if $\alg$ is a $C^*$-algebra then a unitary matrix $U \in M_n(\alg)$ is called a \emph{magic unitary} if each entry
of $U$ is a projection. A unitary $U \in M_n(\alg)$ is called \emph{cubic} if its entries are selfadjoint and the products of
different entries lying in the same row or column are $0$. The following definitions come respectively from \cite{Wang} and
\cite{hyperoct}.

\begin{deft}\label{unisymm}
Denote by  $A_s(n)$ the universal $C^*$-algebra generated by the entries of an  $n$ by $n$ magic unitary $U$. When $U \in M_n \ot
A_s(n)$ is interpreted as the fundamental unitary corepresentation, we  view $A_s(n)$ as the algebra of continuous functions on the quantum
permutation group on $n$ elements,  $S^+_n$.
\end{deft}

\begin{deft}\label{unihyp} Denote by $A_h(n)$ the universal $C^*$-algebra generated by the entries of an  $n$ by $n$ cubic unitary $U$.
When $U \in M_n \ot
A_h(n)$ is interpreted as the fundamental unitary corepresentation, we  view $A_h(n)$ as the algebra of continuous functions on the
quantum hyperoctahedral group on $n$ coordinates, $H^+_n$.
\end{deft}

The quantum groups $O^+_n$, $S^+_n$ and $H^+_n$ are also called the free orthogonal quantum group, the free symmetric quantum
group and the free hyperoctahedral quantum group and can be respectively viewed as liberations of the compact groups $O_n$,
$S_n$ and $H_n$ (\cite{TeoRoland}). More information about their properties, including their interpretations as quantum symmetry
groups can be found in that paper and also respectively in \cite{TeoDeb}, \cite{Teosurvey} and  \cite{hyperoct}.

The following definition was introduced in \cite{TeoRoland}.

\begin{deft}\label{bistoch}
Let $n \in \bn$. Denote by  $A_b(n)$ the universal $C^*$-algebra generated by the entries of an $n$ by $n$ unitary $U$, which
has selfadjoint entries which sum to $1$ in each row/column. When $U \in M_n \ot A_b(n)$ is interpreted as the fundamental
unitary corepresentation, we view $A_b(n)$ as the algebra of continuous functions on the quantum bistochastic group in $n$ dimensions,
$B^+_n$.
\end{deft}

The condition on the sum of entries in each row/column being equal to $1$ is equivalent to stating that the vector
$[1,\ldots,1]^t$ is an eigenvector for both $U$ and $U^t$.  This observation leads to a natural isomorphism $B^+_n \approx
O^+_{n-1}$ (\cite{Raum}).

Relations between the $C^*$-algebras and quantum groups introduced above can be expressed via the following diagrams (arrows on
the level of algebras denote surjective unital $^*$-homomorphisms intertwining the respective coproducts):
\begin{equation} \label{diag1}\begin{matrix}
A_o(n)&\to&A_b(n)\\
\\
\downarrow&&\downarrow\\
\\
A_h(n)&\to&A_s(n)
\end{matrix}
\;\;\;\;\;\;\;\;\;\;\;\;\;\;\;\;\;\;\;\;
\begin{matrix}
O_n^+&\supset&B_n^+\\
\\
\cup&&\cup\\
\\
H_n^+&\supset&S_n^+
\end{matrix}\end{equation}

The diagram on the right suggests a number of  questions, for instance whether $O_n^+=<B_n^+,H_n^+>$, or whether $S_n^+=B_n^+\cap
H_n^+$  (at the level of classical versions, the answers are yes and yes).  Once these questions are properly formulated, for
instance in terms of tensor categories, the answers turn out to be positive, and can be deduced from \cite{TeoRoland}.
We will describe the details later, when we discuss similar problems in a more general 2-parameter framework.

\section{Quantum groups $O^+(p,q)$, $B^+(p,q)$ and $S^+(p,q)$} \label{OBS}
In this section we describe the  two-parameter  families of quantum groups generalising the quantum orthogonal, bistochastic and
symmetric groups described in Section \ref{Secdefcom}. We begin by introducing necessary notations.

Let $p, q \in \bn_0:=\bn \cup \{0\}$, $p^2+q^2>0$  and let $\rho= \tu{e}^{\frac{2 \pi i}{8}}$.   Put $\mathcal{F}= \left(\begin{array}{cc}0 & 1 \\
1 & 0\end{array}\right)$, $\mathcal{C}=\frac{1}{\sqrt{2}} \left(\begin{array}{cc}\rho & \rho^7 \\ \rho^3 &
\rho^5\end{array}\right)$, let $F_{p,0} \in M_{2p}$ be the matrix given by
\[ F_{p,0} = \left( \begin{array}{cccc} \mathcal{F}& 0& \cdots&0 \\ 0& \mathcal{F}& \cdots&0 \\ \vdots & \vdots & \ddots & \vdots \\ 0&0& \ldots
&\mathcal{F}
\end{array} \right)\]
and define $F_{p,q} \in M_{2p+q}$ by
\begin{equation} \label{Fpq}F_{p,q} = \left( \begin{array}{cc} F_{p,0} & 0\\
0 &I_q\end{array} \right)\end{equation} The matrix $F_{p,q}$ is a selfadjoint unitary, and a permutation matrix; moreover
$\overline{F_{p,q}} = (F_{p,q})^t =(F_{p,q})^*$. The matrix $C_{p,q}\in M_{2p+q}$ is defined in an analogous way, replacing
$\mathcal{F}$ by $\mathcal{C}$ in appropriate matrix blocks. Note that $\mathcal{C}$, so also $C_{p,q}$, is unitary.

Whenever we consider matrices of the size $2p+q$ we will denote the indices corresponding to the `$p$-part' by pairs $i\alpha$,
where $i\in\{0,1\}$ and $\alpha\in\{1,\ldots,p\}$ and to the `$q$-part' by  capital Latin letters running from $1$ to $q$.
Moreover we will use the notation $\overline{i}$  (where $\bar{0} =1$, $\bar{1}=0$, so that $\overline{\overline{i}}=i$). This
facilitates the description of the fact that coordinates in the `$p$-part' come naturally in pairs. To simplify the notation we
will also write $\Jnd_p=\{i\alpha:i\in \{0,1\},\alpha\in\{1, \ldots, p\}\}$, $\Jnd_{p,q} = \Jnd_p \cup \{1,\ldots,q\}$, and
extend the `bar' notation to indices in $\Jnd_{p,q}$, writing $\bar{z}=\ib \alpha$ if $z = i\alpha \in \Jnd_p$ and $\bar{z}=M$ if
$z=M\in \{1, \ldots,q\}$. The canonical basis in $\bc^{2p+q}$ will be often denoted by $(e_z)_{z \in \Jnd_{p,q}}$.


\subsection{Quantum group $O^+(p,q) \approx O^+_{2p+q}$}

As  the matrix $F_{p,q}$ defined in \eqref{Fpq} satisfies the conditions listed in Definition \ref{uniorth} we can consider
$A_o(p,q):= A_o(F_{p,q})$  and $O^+(p,q):= O^+(F_{p,q})$. Clearly $O^+(0,q) \approx O^+_q$; in fact the discussion in Section 5
of \cite{coact} implies that $O^+(p,q)\approx O^+_{2p+q}$. From our point of view it is important to consider the following
rephrasing of the above definitions:

\begin{tw}\label{thmOpq}
The algebra $A_o(p,q)$ is the universal $C^*$-algebra generated by elements $\{U_{z,y}: z,y \in \Jnd_{p,q}\}$ such that the
resulting $2p+q$ by $2p+q$ matrix $U$ is unitary and for each $i\alpha, j\beta \in \Jnd_p, M,N \in \{1, \ldots,q\}$ we have
\begin{equation} \label{ad1}  U_{i \alpha, j \beta}^* = U_{\bar{i} \alpha, \bar{j} \beta},\end{equation}
\begin{equation} \label{ad2} U_{i \alpha, N}^* = U_{\bar{i} \alpha, N},\end{equation}
\begin{equation} \label{ad3} U_{M, j \beta}^* = U_{M, \bar{j} \beta}, \end{equation}
\begin{equation} \label{ad4} U_{M,N}^* = U_{M,N}. \end{equation}
Moreover $A_o(p,q)$ with $U$ viewed as a fundamental corepresentation is the algebra of continuous functions on a compact quantum group
$O^+(p,q)\approx O^+_{2p+q}$.
\end{tw}

\begin{proof}
The first part is a direct consequence of the fact that $F_{p,q}$ has a block-matrix form and formulas:
\[ \begin{bmatrix}0 & 1 \\ 1 & 0  \end{bmatrix}  \begin{bmatrix}A^* & B^* \\ C^*  & D^*  \end{bmatrix}
= \begin{bmatrix}C^* & D^* \\ A^*  & B^*  \end{bmatrix}, \;\;\; \begin{bmatrix}A^* & B^* \\ C^*  & D^*  \end{bmatrix}
\begin{bmatrix}0 & 1 \\ 1 & 0  \end{bmatrix} = \begin{bmatrix}B^* & A^* \\ D^*  & C^*  \end{bmatrix} \] and
\[ \begin{bmatrix}0 & 1 \\ 1 & 0  \end{bmatrix}  \begin{bmatrix}A^* & B^* \\ C^*  & D^*  \end{bmatrix}  \begin{bmatrix}0 & 1 \\ 1 & 0  \end{bmatrix}
= \begin{bmatrix}D^* & C^* \\ B^*  & A^*  \end{bmatrix}, \]
which imply that if $U = (U_{z,y})_{z,y \in \Jnd_{p,q}}$ then $U = (F_{p,q} \ot 1) \bar{U} (F_{p,q} \ot 1)$ if and only if the entries of $U$ satisfy relations \eqref{ad1}-\eqref{ad4}.

The second part follows from the discussion before the theorem, but can be also seen directly: indeed, exploiting the equalities of the type
\[ \begin{bmatrix}\rho & \rho^7 \\ \rho^3 & \rho^5  \end{bmatrix}  \begin{bmatrix}A 
\\ C  
\end{bmatrix}
= \begin{bmatrix}\rho(A + \rho^6 C) 
 \\ \rho^3 (A + \rho^2C)  
\end{bmatrix}, \]
 \[ \begin{bmatrix}A & B
  \end{bmatrix} \begin{bmatrix}\rho^7 & \rho^5 \\ \rho & \rho^3  \end{bmatrix}
= \begin{bmatrix}\rho(\rho^6 A + B) & \rho^3(\rho^2 A + B) 
\end{bmatrix}  \] and
\[ \begin{bmatrix}\rho & \rho^7 \\ \rho^3 & \rho^5  \end{bmatrix}  \begin{bmatrix}A & B \\ C  & D  \end{bmatrix}
\begin{bmatrix}\rho^7 & \rho^5 \\ \rho & \rho^3  \end{bmatrix}  = \begin{bmatrix}A -i C + iB + D & -iA -C - B + iD \\
 iA - C -B -iD & A + iC - iB + D  \end{bmatrix}
 \]
one can check that if
$\alg$ is a $C^*$-algebra and $U \in M_{2p+q} (\alg)$ is a unitary matrix, then the entries of $U$ satisfy conditions displayed
in the theorem if and only if the entries of the unitary matrix $(C_{p,q} \ot 1_{\alg}) U ({C_{p,q}}^* \ot 1_{\alg})$ are
self-adjoint.
\end{proof}

Note that in terms of the notation introduced earlier the relations \eqref{ad1}-\eqref{ad4} can be summarized by saying
\begin{equation} \label{ad5} U_{z,y} = U^*_{\bar{z}, \bar{y}}, \;\;\;z,y \in \Jnd_{p,q}.\end{equation}

It is well known that algebraic relations between entries of the fundamental representation $U$ of a given quantum group can be interpreted as declaring certain scalar matrices to be elements of $\Hom (U^{\ot k}; U^{\ot l})$ for some $k,l \in \bn$ (see for
example \cite{Teoold}).

\begin{prop} \label{HomOpq}
The algebra $A_o(p,q)$ is the universal $C^*$-algebra generated by the entries of a unitary $2p+q$ by $2p+q$ matrix $U$ such that the
vector $\xi:= \sum_{i\alpha \in \Jnd_p} e_{i\alpha} \ot e_{\ib \alpha} + \sum_{M=1}^q e_M \ot e_M$ is a fixed vector for $U^{\ot
2}$ (in other words the map $1 \to \xi$ is  an element of $\Hom (1;U^{\ot 2})$).
\end{prop}

\begin{proof}
Let $(U_{z,y})_{z,y \in \Jnd_{p,q}}$ be unitary. Then $U^{\ot 2} \xi = \xi$ if and only if
\begin{align*}
\sum_{z, y \in \Jnd_{p,q}} \left(\sum_{i\alpha \in \Jnd_p} U_{z,i\alpha} U_{y,\ib \alpha} + \sum_{M=1}^q U_{z,M} U_{y,M} \right)
e_z \ot e_y = \xi.
\end{align*}
This means that for example for each $j\beta \in \Jnd_q$ and $y \in \Jnd_{p,q}$ we must have
\[ \sum_{i\alpha \in \Jnd_p}
U_{j\beta,i\alpha} U_{y,\ib \alpha} + \sum_{M=1}^q U_{j\beta,M} U_{y,M} = \delta_{y}^{\bar{j}\beta}.\] Fix $z \in \Jnd_{p,q}$,
multiply the last formula by $U_{y,z}^*$ on the right and sum over $y \in \Jnd_{p,q}$. As $U$ is unitary, this yields
\[\sum_{i\alpha \in \Jnd_p}
U_{j\beta,i\alpha} \delta_{\ib \alpha}^z  + \sum_{M=1}^q U_{j\beta,M} \delta_{M}^z = U_{\bar{j} \beta, z}^* .\] Considering
respectively $z\in \{1, \ldots, q\}$ and $z \in \Jnd_p$ yields the first two displayed formulas in Theorem \ref{thmOpq}. It is
easy to see that the third and fourth formula can be obtained in an analogous way.
\end{proof}

The last proposition allows one to describe the  tensor category corresponding to the dual of
$O^+(p,q)$ in terms of Temperley-Lieb diagrams (\cite{Teoold}) or noncrossing pair partitions (\cite{TeoRoland}). We refer to
these papers for details; in Section \ref{coreps-partitions} below we state the corresponding results in terms of partitions.

\subsection{Quantum group $B^+(p,q) \approx O^+_{p+q-i(p,q)}$}

\begin{deft}\label{Bpq}
The algebra $A_b(p,q)$ is the universal $C^*$-algebra generated by elements $\{U_{z,y}, z,y \in \Jnd_{p,q}\}$ which satisfy all the
relations required of generators of $A_o(p,q)$ and additionally are such that entries in each row/column of the resulting unitary $U$
sum to $1$: for all $y \in \Jnd_{p,q}$
\[  \sum_{z \in \Jnd_{p,q}} U_{z,y} = \sum_{z \in \Jnd_{p,q}} U_{y,z}=1.\]
\end{deft}

Let $\eta = \sum_{z \in \Jnd_{p,q}} e_z \in \bc^{2p+q}$ and recall the matrix $C_{p,q}$ defined in the beginning of this section.

\begin{tw} \label{Bpqpreserv}
The algebra $A_b(p,q)$ is the algebra of continuous functions on a compact quantum group, denoted further $B^+(p,q)$. The unitary
$U=(U_{z,y})_{ z,y \in \Jnd_{p,q}} \in M_{2p+q} \ot A_b(p,q)$ is the fundamental representation of $B^+(p,q)$. The algebra
$A_b(p,q)$ is isomorphic to the universal $C^*$-algebra generated by entries of a $2p+q$ by $2p+q$ unitary $V$ which is orthogonal and
satisfies the condition $V (C_{p,q} \eta \ot 1) = C_{p,q} \eta \ot 1$.
\end{tw}

\begin{proof}
As explained after Definition \ref{bistoch} the condition that the entries in each row and column of a matrix $U$ sum to
$1$ are equivalent to the fact that the vector $\eta$ is fixed both by $U$ and $U^t$. This observation (or a direct computation)
implies that $A_b(p,q)$ is the algebra of continuous functions on a compact quantum group. Further note that as $F_{p,q}
\eta=\eta$ and $F_{p,q}$ is a selfadjoint unitary we have the following string of equivalences for a unitary $U\in M_n
(A_o(p,q))$ satisfying the condition \eqref{UFbar} with $F=F_{p,q}$:
\begin{align*} U (\eta \ot 1) = \eta
\ot 1 & \;\Leftrightarrow \; U^*(\eta \ot 1) = \eta \ot 1 \;\Leftrightarrow \; (F_{p,q} \ot 1) U^t (F_{p,q} \ot 1) (\eta \ot 1) =
\eta \ot 1 \\&\;\Leftrightarrow  U^t (\eta \ot 1) = \eta \ot 1, \end{align*} so that the condition on the sum of entries in each
column of a unitary $U$ as above being equal to $1$ follows from the analogous condition for rows.

In the last part of the proof of Theorem \ref{thmOpq} we noticed that the transformation between the fundamental unitary in
$O^+(p,q)$ and that of $O^+_{2p+q}$ is implemented by conjugating with the unitary matrix $C_{p,q}$. Hence to prove the last
statement it suffices to note that a unitary $U$ fixes the vector $\eta$ if and only if $(C_{p,q}\ot 1) U (C_{p,q}^*\ot 1)$ fixes
$C_{p,q} \eta$.
\end{proof}

When $q=0$ or $p=0$ the above theorem implies that the `deformed' quantum bistochastic group coincides with the one studied in
\cite{TeoRoland}.

\begin{tw}
The following isomorphisms hold:
\[ B^+(p,0)\approx B^+_{2p} \approx O^+_{2p-1}, \;\;\; B^+(0,q)\approx B^+_{q} \approx O^+_{q-1},\]
\[ B^+(p,q)\approx O^+_{2p+q-2} \textrm{ for } pq>0,\]
\[ B^+(p,1) \approx B^+(p,0).\]
\end{tw}
\begin{proof}
The only isomorphism in the first line that needs to be proved is  $B^+(p,0)\approx B^+_{2p}$ (the ones involving the quantum orthogonal
groups are the consequences of \cite{Raum}, as explained after Definition \ref{bistoch}). Due to Theorem \ref{Bpqpreserv} the
fundamental representation of $B^+(p,0)$ can be defined as a $2p$ by $2p$ unitary matrix $U$ with selfadjoint entries which
satisfies the condition $U (C_{p,0} \eta \ot 1) = C_{p,0} \eta \ot 1$. It is clear that in the above condition $C_{p,0} \eta $
can be replaced by a non-zero scalar multiple; in particular by the vector $[1, -1, \ldots, 1, -1]^{t}$. But that vector can be
mapped by a real orthogonal matrix onto $\eta$, so also onto $[1, 0,0, \ldots, 0]^{t}$, and the argument of \cite{Raum} implies
that the desired isomorphism holds.

Assume now that $pq>0$. Using once again Theorem \ref{Bpqpreserv} a fundamental representation of $B^+(p,q)$ can be defined
as a $2p+q$ by $2p+q$ unitary matrix $U$ with selfadjoint entries which satisfies the condition $U (C_{p,q} \eta \ot 1) = C_{p,q}
\eta \ot 1$. One can check that the vector $C_{p,q} \eta$ is proportional to the vector $\wt{\eta}_{p,q}\in \bc^{2p+q}$ given by
\begin{align*} \wt{\eta}_{p,q}&= [\underbrace{1-i, i-1,\ldots, 1-i, i-1}_{2p \textup{ times}},\underbrace{1+i, \ldots,1+i}_{q \textup{
times }}]^T\\ &= [\underbrace{1, -1,\ldots, 1,-1}_{2p \textup{ times}},\underbrace{1, \ldots,1}_{q \textup{ times }}]^T + i
[\underbrace{-1, 1,\ldots, -1, 1}_{2p \textup{ times}},\underbrace{1, \ldots,1}_{q \textup{ times }}]^T   \end{align*} As $U$ has
selfadjoint entries, it preserves $\wt{\eta}_{p,q}$ if and only if it preserves its real and imaginary parts; equivalently, it
preserves vectors \[  [\underbrace{1, -1,\ldots, 1, -1}_{2p \textup{ times}},\underbrace{0, \ldots,0}_{q \textup{ times }}]^T
\] and
\[  [\underbrace{0, 0,\ldots, 0, 0}_{2p \textup{ times}},\underbrace{1, \ldots,1}_{q \textup{ times
}}]^T. \] Repeating an earlier argument we find a matrix in $O_{2p} \times O_q \subset O_{2p+q}$ mapping these vectors
respectively to $[1, 0, \ldots,0]^T$ and $[0,\ldots,0,1]^T$; this provides the isomorphism $B^+(p,q)\approx O^+_{2p+q-2}$, which
implies in particular that $B^+(p,1)\approx B^+_{p,0}$.
\end{proof}

Note that the argument in the second part of the above proof can be framed in general terms -- for $n \geq 3$ the group of $n$ by
$n$ orthogonal matrices preserving a fixed non-zero vector $v\in \bc^n$  is either isomorphic to $O_{n-1}$ (when the real and
imaginary parts of $v$ are proportional to each other) or to $O_{n-2}$ (when they are not); the same applies to $O^+_n$.

The following result is a consequence of Proposition \ref{HomOpq} and the arguments in the proof of Theorem \ref{Bpqpreserv}.

\begin{cor} \label{HomBpq}
The algebra $A_b(p,q)$ is the universal $C^*$-algebra generated by the entries of a unitary $2p+q$ by $2p+q$ matrix $U$ such that the
vector $\xi$ defined in Proposition \ref{HomOpq} is a fixed vector for $U^{\ot 2}$ and the vector $\eta = \sum_{z \in \Jnd_{p,q}}
e_z$ is a fixed vector for $U$.
\end{cor}

The above corollary implies that the category of representations of $B^+(p,q)$ coincides with that of $B^+(2p+q)$
(\cite{TeoRoland}), see Section \ref{coreps-partitions}.

\subsection{Quantum group $S^+(p,q)\approx H^+_p \hat{\star} \, S^+_q$}

The free quantum group of permutations of $n$ elements may be viewed as the universal $C^*$-algebra generated by entries of an
orthogonal matrix which are additionally required to be orthogonal projections. This motivates the following definition.

\begin{deft}\label{Spq}
The algebra $A_s(p,q)$ is the universal $C^*$-algebra generated by projections $\{U_{z,y}, z,y \in \Jnd_{p,q}\}$ which satisfy all the
relations required of generators of $A_o(p,q)$.
\end{deft}

As for every  magic unitary entries lying in the same row or column are pairwise orthogonal, the generators of $A_s(p,q)$ satisfy the following relations:
\[ U_{z,y}^* U_{z,x} = U_{z,y} U_{z,x} = 0, \;\;U_{y,z}^* U_{x,z} = U_{y,z} U_{x,z}=0 \;\;z,y,x \in \Jnd_{p+q}, x \neq y.\]

\begin{prop}\label{directspq}
The algebra $A_s(p,q)$ is the universal $C^*$-algebra generated by two families of projections $\{U_{i\alpha,j\beta}:i\alpha, j\beta
\in\Jnd_p\}$ and $\{U_{M,N};M,N \in \{1, \ldots,q\}\}$, such that both matrices $(U_{i\alpha,j\beta})_{i\alpha, j\beta
\in\Jnd_p}$ and $(U_{M,N})_{M,N=1}^q$ are magic unitaries and moreover
\begin{equation} U_{i \alpha, j \beta} = U_{\bar{i}
\alpha, \bar{j} \beta}, \;\;\; i\alpha, j\beta \in \Jnd_p. \label{adjointhp}\end{equation}
\end{prop}

\begin{proof}
It suffices to show that whenever $i\alpha \in \Jnd_p$ and $M \in \{1, \ldots,q\}$ then $U_{i \alpha, M} = 0 = U_{M, i\alpha}$.
But the matrix $\{U_{z,y}, z,y \in \Jnd_{p,q}\} \in M_{2p+q}(A_s(p,q))$ is a magic unitary, so each of its columns consists of
mutually orthogonal projections, and as we have $U_{i \alpha, M} = U_{\bar{i} \alpha, M}$, it follows that $U_{i \alpha, M}=0$.
The second equality follows from the orthogonality of projections in each row of a magic unitary.
\end{proof}

\begin{tw} \label{thSpq}
The algebra $A_s(p,q)$ is the algebra of continuous functions on a compact quantum group, denoted further $S^+(p,q)$. The unitary
$U=(U_{z,y})_{ z,y \in \Jnd_{p,q}} \in M_{2p+q} \ot A_s(p,q)$ is the fundamental representation of $S^+(p,q)$. The algebra
$A_s(p,q)$ is isomorphic to the free product $A_h(p) \star A_s(q)$; on the level of quantum groups we have $S^+(p,q)\approx H^+_p
\, \hat{\star} \, S^+_q$.
\end{tw}

\begin{proof}
Due to Proposition \ref{directspq} it suffices to consider separately the cases $p=0$ and $q=0$; as we have $A_s(0,q) \approx A_s(q)$
(with the isomorphism preserving natural fundamental corepresentations), we can assume that $q=0$. Further it suffices to show
that if we define for each $i\alpha, j\beta \in \Jnd_p$
\[ \tilde{U}_{i\alpha, j\beta} := \sum_{k\gamma \in \Jnd_p} U_{i\alpha, k\gamma} \ot U_{k\gamma, j\beta}\] then
each $\tilde{U}_{i\alpha, j\beta}$ is a projection (the fact that the conditions in \eqref{adjointhp} will then be satisfied
follows from Theorem \ref{thmOpq}; and analogous statements for the potential antipode and counit follow from the fact that the
adjoint of a projection is a projection and $0,1\in \bc$ are projections). The last statement is however a direct consequence of
orthogonality of the rows/columns of the magic unitary; it can also be deduced from the fact that the map $U_{i\alpha, j\beta} \to \tilde{U}_{i\alpha, j\beta}$ is a $^*$-homomorphism.

Proposition \ref{directspq} implies also that $A_s(p,q) \approx A_s(p,0) \star A_s(0,q)$. Therefore it suffices to show that
$A_s(p,0) \approx A_h(p)$. This is however an immediate consequence of the fact that $A_h(p)$ can be defined via the requirement
that its fundamental corepresentation is a $2p$ by $2p$ \emph{sudoku} unitary, i.e.\ a magic unitary which has a block matrix
form $\left(\begin{array}{cc}a & b \\b &a \end{array}\right)$ (Definition 5.2 and Theorem 6.2 of \cite{hyperoct}). Indeed, formulas \eqref{ad1}-\eqref{ad4} imply that the fundamental unitary representation of $S^+(p,0)$ is a magic unitary of the form
\[ \begin{bmatrix} A& B & C&D & E&F & \cdots \\ B&A& D &C & F &E & \cdots\\G&H & I &J & K & L & \cdots
\\ H&G&J&I&L&K& \cdots \\ M&N&O&P&Q&R & \cdots \\N&M&P&O&R&Q & \cdots \\ \vdots & \vdots & \vdots & \vdots & \vdots &\vdots & \ddots
\end{bmatrix}, \]
which can be transformed into a sudoku unitary by permuting rows and columns (so that odd rows/columns remain in the same order but are shifted to the left/up so that they become first $p$ rows/columns).
\end{proof}

The next proposition facilitates the description of the category of the representations of $S^+(p,q)$ in terms of the noncrossing partitions, see Section \ref{coreps-partitions}.

\begin{prop} \label{HomSpq}
The algebra $A_s(p,q)$ is the universal $C^*$-algebra generated by the entries of a unitary $2p+q$ by $2p+q$ matrix $U$ such that the
vector $\xi$ defined in Proposition \ref{HomOpq} is a fixed vector for $U^{\ot 2}$ and the map $e_{i\alpha} \to e_{i\alpha} \ot
e_{\bar{i}\alpha}$, $e_M \to e_{M} \ot e_{M}$ defines a morphism in $\Hom (U; U^{\ot 2})$.
\end{prop}

\begin{proof}
If $U$ is a unitary $2p+q$ by $2p + q$ matrix such that $\xi$ is a fixed vector for $U^{\ot 2}$, then by Proposition
\ref{HomOpq} the entries of $U$ satisfy the relations \eqref{ad1}-\eqref{ad4}. Further the condition that the map described in this proposition is an intertwiner between $U$ and $U^{\ot 2}$ is satisfied if and only if for each $i\alpha \in \Jnd_p$
\[ \sum_{j\beta \in \Jnd_p} e_{j\beta} \ot e_{\overline{j} \beta} \ot U_{j\beta, i\alpha} +
\sum_{N=1}^q e_{N} \ot e_{N} \ot U_{N, i\alpha} = \sum_{z,y \in \Jnd_{p+q}}
e_z \ot e_y \ot U_{z, i\alpha} U_{y, \ib\alpha}\]
and  for each  $M \in \{1, \ldots, q\}$
\[ \sum_{j\beta \in \Jnd_p} e_{j\beta} \ot e_{\bar{j} \beta} \ot U_{j\beta, M} +
\sum_{N=1}^q e_{N} \ot e_{N} \ot U_{N, M} = \sum_{z,y \in \Jnd_{p+q}} e_z \ot e_y \ot U_{z, M} U_{y, M}.\] The above conditions
hold if and only if for each $i\alpha, j \beta \in \Jnd_p, M,N \in \{1, \ldots, q\}$, $y \in \Jnd_{p+q}$
\[ U_{j\beta, i \alpha} U_{y, \ib \alpha} = \delta_y^{\overline{j} \beta} U_{j\beta, i \alpha},\]
\[ U_{N, i \alpha} U_{y, \ib \alpha} = \delta_y^{N} U_{N, i \alpha},\]
\[ U_{j\beta, M} U_{y, M} = \delta_y^{\overline{j} \beta} U_{j\beta, M},\]
\[ U_{N, M} U_{y, M} = \delta_y^{N} U_{N, M}.\]
Bearing in mind the relations \eqref{ad1}-\eqref{ad4} we see that  all entries of $U$ are orthogonal projections; hence $U$ is a
magic unitary. Conversely it is easy to check that if $U$ is a magic unitary whose entries satisfy \eqref{ad1}-\eqref{ad4} then
the last four displayed formulas automatically hold (see also the comment after Definition \ref{Spq}).
\end{proof}

\section{Quantum group $H^+(p,q)$}

The quantum groups we defined in the last section have all been shown to be closely related to
well-studied objects. The generalized quantum hyperoctahedral groups to be introduced here are genuinely new quantum groups,
connected to quantum hyperoctahedral groups (\cite{hyperoct}) and quantum isometry groups of $C^*$-algebras of free groups
(\cite{groupdual}).

The fact that an element $x$ of a $C^*$-algebra is an orthogonal projection can be written as $x=x^*x$. It is natural to consider
the condition $x=x x^* x$, which of course means that  $x$ is a partial isometry. This leads to the following definition.

\begin{deft}\label{Hpq}
The algebra $A_h(p,q)$ is the universal $C^*$-algebra generated by partial isometries $\{U_{z,y}:z,y \in \Jnd_{p+q}\}$ which satisfy all the relations required of generators of $A_o(p,q)$. 
\end{deft}

Here we also have some automatic `orthogonality', which will be described by the next proposition and its corollary. 

\begin{prop}\label{ort0}
Let $\alg$ be a $C^*$-algebra, $n \in \bn$ and let $U\in M_{n}(\alg)$ be a unitary matrix whose entries are partial isometries. Then
\[ U_{y,z} U_{x,z}^* = U_{z,y}^* U_{z,x} = 0, \;\;\;\;z,y,x \in \{1, \ldots,n\}, x \neq y.\]
\end{prop}

\begin{proof}
For each $y,z \in \{1,\ldots,n\}$ denote the initial and range projections of $U_{y,z}$
respectively by $P_{y,z}$ and $Q_{y,z}$, so that
\[ P_{y,z} = U_{y,z}^* U_{y,z}, \;\;\; Q_{y,z} = U_{y,z} U_{y,z}^*.\]
Fix $z \in \{1, \ldots,n\}$. The unitarity of $U$ implies that $\sum_{y=1}^n Q_{z,y} = \sum_{y=1}^n P_{y,z}=1_{\alg}$, so that
for $y,x \in \{1, \ldots,n\}, x \neq y$ there is  $Q_{z,y} Q_{z,x} = P_{y,z} P_{x,z} =0$. The initial/range projection
interpretation ends the proof.
\end{proof}

Application of the above proposition and the equality \eqref{ad5} immediately gives the following corollary.

\begin{cor} \label{corort}
Let $U\in M_{2p+q}(\alg_h(p,q))$ be the unitary matrix `generating' $\alg_h(p,q)$.   Then
\[ U_{y,z} U_{x,z}^* = U_{z,y}^* U_{z,x} = U_{y,z}^* U_{x,z} = U_{z,y} U_{z,x}^*= 0, \;\;\;\;z,y,x \in \Jnd_{p,q}, x \neq y.\]
\end{cor}

The last observations lead to the following theorem.

\begin{tw}
The algebra $A_h(p,q)$ is the algebra of continuous functions on a compact quantum group, denoted further $H^+(p,q)$. The unitary
$U=(U_{z,y})_{ z,y \in \Jnd_{p,q}} \in M_{2p+q} \ot A_h(p,q)$ is the fundamental representation of $H^+(p,q)$. The quantum
group $H^+(0,q)$ is the quantum hyperoctahedral group $H^+_q$ studied in \cite{hyperoct}.
\end{tw}

\begin{proof}
The first statement can be deduced as in the proof of Theorem \ref{Spq}, using the fact that a $^*$-homomorphic image (and the
adjoint) of a partial isometry is a partial isometry.  It remains to observe that the algebra $A_h(0,q)$ is isomorphic to
$A_h(q)$. But this is a natural consequence of Corollary \ref{corort} and the fact that if a unitary is cubic then its entries satisfy
the condition $u_{ij} = u_{ij}^3 =u_{ij}^*$ (easy to show and noted implicitly in \cite{hyperoct}).
\end{proof}

As stated in the introduction, the quantum group $H^+(p,0)$ is in fact the quantum isometry group of the dual of the free group
$\mathbb{F}_p$ (\cite{groupdual}). Let us quickly recall the general notion of a quantum isometry group of the dual of a finitely
generated discrete group.

Let $\Gamma$ be a finitely generated discrete group with a fixed finite symmetric set of generators $S \subset \Gamma$. The
choice of a generating set $S$ determines a word-length function $l$ on $\Gamma$. Denote the universal group $C^*$-algebra of
$\Gamma$ by $C(\widehat{\Gamma})$ and let the group ring $\bc[\Gamma] \subset C(\widehat{\Gamma})$ be denoted by
$R(\widehat{\Gamma})$. Then the multiplication by the length function defines an operator $\hat{D}:R(\widehat{\Gamma}) \to
R(\widehat{\Gamma})$,
\[ \hat{D}(\lambda_{\gamma}) = l(\gamma) \lambda_{\gamma}, \;\;\; \gamma \in \Gamma.\]
We say that a quantum group $\QG$ \emph{acts on the dual of $\Gamma$ by orientation preserving isometries} if there exists a
unital $^*$-homomorphism $\alpha: C(\widehat{\Gamma}) \to C(\widehat{\Gamma}) \ot C(\QG)$ such that
\[ (\alpha \ot \id_{C(\QG)}) \alpha = (\id_{C(\widehat{\Gamma})} \ot \Com) \alpha\]
and moreover $\alpha$ restricts to a unital $^*$-homomorphism $\alpha_0: R(\hat{\Gamma}) \to R(\hat{\Gamma}) \odot R(\QG)$ satisfying
the commutation relation
\begin{equation} \label{isom} \alpha_0 \hat{D} = (\hat{D} \ot \id_{R(\QG)}) \alpha_0\end{equation}
and preserving the canonical trace on $R(\widehat{\Gamma})$.  For the motivation behind this definition and connections with
spectral triples and noncommutative manifolds we refer to \cite{groupdual} and references therein.

\begin{tw}[\cite{groupdual}]
Let $\Gamma$ be a discrete group with a fixed finite symmetric set of generators $S$. The category of all compact quantum groups
acting on $\widehat{\Gamma}$ by orientation preserving isometries admits a universal (initial) object, denoted further by
$QISO^+(\widehat{\Gamma}, S)$ and called the quantum group of orientation preserving isometries of $\widehat{\Gamma}$.
\end{tw}

When the choice of the generating set is clear, we write simply $QISO^+(\widehat{\Gamma})$. In particular if $\Gamma=\mathbb{F}_p$, the free
group on $p$ generators $x_1, \ldots, x_p$, we use the generating set $S=\{x_1, x_1^{-1}, \ldots, x_p, x_p^{-1}\}$. The following result is
essentially a rephrasing  of Theorem 5.1 of \cite{groupdual}; for the convenience of the reader we sketch the proof.

\begin{tw}
The quantum group $H^+(p,0)$ is isomorphic to the quantum group of orientation preserving isometries of
$\widehat{\mathbb{F}_p}$.
\end{tw}

\begin{proof}
Discussion after Theorem 2.6 of \cite{groupdual} implies that $\QG:=QISO^+(\mathbb{F}_p)$ is a compact matrix quantum group with
a fundamental representation determined by a unitary $U = (U_{t,s})_{t,s \in S} \in M_{2p} (C(\QG))$ such that the map
\[ \alpha_0 (\lambda_t) = \sum_{s \in S} \lambda_s \ot U_{s,t}, \;\; t \in S,\]
extends to a $^*$-homomorphism $\alpha_0: R(\widehat{\mathbb{F}_p}) \to R(\widehat{\mathbb{F}_p}) \odot R(\QG)$ satisfying
\eqref{isom}. Relabel generators in $S$ so that $S=\{x_{i\alpha}, i\alpha \in \Jnd_p\}$ and $x_{i\alpha} = x_{\ib \alpha}^{-1}$
for each $i\alpha \in \Jnd_p$. Then the fact that $\alpha_0$ is a $^*$-map implies that the entries of $U$ satisfy the relations
\eqref{ad1} (where we write $U_{i\alpha, j \beta}:= U_{x_{i\alpha}, x_{j\beta}}$). Further unitality of $\alpha_0$ (specifically
the conditions $\alpha_0(x_t) \alpha_0(x_{t^{-1}}) = 1_{C(\widehat{\mathbb{F}_p})} \ot 1_{\QG}$ for each $t \in S$) together with
unitarity of $U$ imply that each $U_{i\alpha}$ is a partial isometry. Finally a combinatorial argument shows that no additional relations are implied by the fact that
the $^*$-homomorphism $\alpha_0$ satisfies \eqref{isom}; hence the universal properties defining the standard
fundamental representations of $\QG$ and $H^+(p,0)$ coincide.
\end{proof}

Finally we describe $A_h(p,q)$ in terms of the intertwiners between the tensor powers of the fundamental corepresentation.

\begin{prop} \label{HomHpq}
The algebra $A_h(p,q)$ is the universal $C^*$-algebra generated by the entries of a unitary $2p+q$ by $2p+q$ matrix $U$ such that the
vector $\xi$ defined in Proposition \ref{HomOpq} is a fixed vector for $U^{\ot 2}$ and the map $e_{i\alpha} \to e_{i\alpha} \ot
e_{\bar{i}\alpha}\ot e_{i\alpha}$, $e_M \to e_{M} \ot e_{M} \ot e_M$ defines a morphism in $\Hom (U; U^{\ot 3})$.
\end{prop}

\begin{proof}
The proof is very similar to that of Proposition \ref{HomSpq}. The difference lies in the fact that we obtain conditions of the type ($i\alpha, j \beta \in \Jnd_p, M,N \in \{1, \ldots, q\}$, $y,x \in \Jnd_{p+q}$)
\[ U_{j\beta, i \alpha} U_{y, \ib \alpha} U_{x, i \alpha} = \delta_y^{\overline{j} \beta} \delta_x^{j \beta} U_{j\beta, i \alpha},\]
\[ U_{N, i \alpha} U_{y, \ib \alpha} U_{x, i \alpha} = \delta_y^{N}  \delta_x^{N} U_{N, i \alpha},\]
\[ U_{j\beta, M} U_{y, M} U_{x, M} = \delta_y^{\overline{j} \beta} \delta_x^{j \beta} U_{j\beta, M},\]
\[ U_{N, M} U_{y, M} U_{x,M} = \delta_y^{N} \delta_x^{N}  U_{N, M}.\]
These together with relations \eqref{ad1}-\eqref{ad4} imply that all $U_{z,y}$ ($z,y \in \Jnd_{p,q}$) are partial isometries. 
\end{proof}

\section{Categories of representations via partitions} \label{coreps-partitions}

In this section we describe the categories of representations of the quantum groups we are considering in this paper in terms
of (marked) partitions.  Let $P(k,l)$ ($k, l \in \bn_0$) be the set of partitions between $k$ upper points and $l$ lower points.
Given $\pi\in P(k,l)$ and two multi-indices $i=(i_1,\ldots,i_k)\in \Jnd_{p,q}^k$ and $j=(j_1,\ldots,j_l)\in \Jnd_{p,q}^l$, we define a number
$\delta_\pi(^i_j)\in\{0,1\}$ in the following way:  first place the indices $(i_1,\ldots,i_k)$ and $(j_1,\ldots,j_l)$ respectively on the upper and lower points and then put $\delta_\pi(^i_j)=1$ if in any block of $\pi$ the upper and lower sequences
of indices in $\Jnd_{p,q}$, say $(x_1,\ldots,x_r)\in \Jnd_{p,q}^r$ and $(y_1,\ldots,y_s)\in \Jnd_{p,q}^s$, satisfy $x_1=\bar{x}_2=x_3=\bar{x}_4=\ldots$, $
y_1=\bar{y}_2=y_3=\bar{y}_4=\ldots$, $x_1=y_1$, and $\delta_\pi(^i_j)=0$ otherwise. Thus $\delta_\pi(^i_j)=1$ if and only if indices in each block of the partition $\pi$ have the following pattern:
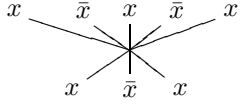
\[ \xymatrix@R=0.5pc @C=0.5pc{  \;\;\; \ar@{-}[drr] x & \;\;\;\bar{x}\ar@{-}[dr]& x \ar@{-}[d]& \bar{x} \;\ar@{-}[dl] &x  \;\; \ar@{-}[dll]\\ & & *=0{}& & \\ & x \ar@{-}[ur]& \bar{x} \ar@{-}[u]& x \ar@{-}[ul] & } \]
Further consider the following operator in $B((\bc^{2p+q})^{\ot k};(\bc^{2p+q})^{\ot l})$:
\begin{equation}\label{Tpi}T_\pi(e_{i_1}\otimes\ldots\otimes e_{i_k})=\sum_{j_1,\ldots,j_l\in \Jnd_{p,q}}\delta_\pi(^i_j)\;e_{j_1}\otimes\ldots\otimes
e_{j_l}, \;\;\;\;\;\;\; (i_1, \ldots, i_k) \in \Jnd_{p,q}^k.\end{equation}

We denote by $P_2,P_{12},P_{even}\subset P$ respectively the pairings, the singletons and pairings, and the partitions having even blocks. Let
also $NC_x=NC\cap P_x$, for any $x\in\{.,2,12,\tu{even}\}$, where $NC\subset P$ denotes the set of all non-crossing partitions. It turns out that such collections of partitions, known to
correspond to representations of the orthogonal, symmetric, bistochastic and hyperoctahedral (quantum) groups
(\cite{TeoRoland}), can be also used to describe the categories in our two-parameter context. Indeed, one can check, similarly as
it was done in \cite{TeoRoland}, that for $P_{12}$ and $P_{\tu{even}}$ the usual operations of tensoring, concatenation (with
the appropriate multiplication factor added for deleted closed blocks), and turning the partition upside-down correspond to tensoring,
composing and passing to the adjoint on the level of the associated operators $T_{\pi}$ defined in \eqref{Tpi}. Note that such a
statement fails when we consider the whole category $P$ -- this explains why the case of $S^+(p,q)$, to be discussed later on, cannot be included in the
following theorem.

\begin{tw} \label{catOHB}
For $\QG=O^+(p,q)$ (respectively, $\QG=B^+(p,q),H^+(p,q)$) let $U$ denote the fundamental representation introduced earlier. Then for all $k,l \in \bn_0$
$$\Hom(U^{\otimes k};U^{\otimes l})=\tu{span}(T_\pi|\, \pi\in D(k,l)),$$
with $D=NC_2$ (respectively, $D=NC_{12},NC_{even}$).
\end{tw}

\begin{proof}

In the one-parameter case ($p=0$) the result was proved in \cite{TeoRoland}; it turns out that the methods used there can be easily adopted to our framework when we work with the quantum groups listed in the theorem.  Details will be provided only for
the case of $O^+(p,q)$, as the other two follow similarly.

In fact for $O^+(p,q)$ the result is a consequence of the general facts about $A_o(F)$. Indeed,  let $F\in M_n(\mathbb R)$ satisfy
$F=F^t$ and $F^2=1$, and consider the algebra $A_o(F)$. Our claim is that we have $\Hom(U^{\otimes k};U^{\otimes
l})=\tu{span}(T_\pi|\pi\in NC_2(k,l))$, where $T_\pi(e_{i_1}\otimes\ldots\otimes
e_{i_k})=\sum_{j_1,\ldots,j_l}\delta_\pi(^i_j)e_{j_1}\otimes\ldots\otimes e_{j_l}$, with $\delta_\pi(^i_j)\in\mathbb R$ being
given by:
$$\delta_\pi\begin{pmatrix}i_1&\ldots&i_k\\ j_1&\ldots&j_l\end{pmatrix}
=\prod_{(i_ri_s)\in\pi}F_{i_ri_s}\prod_{(j_rj_s)\in\pi}F_{j_rj_s}\prod_{(i_rj_s)\in\pi}\delta_{i_rj_s}$$ In other words,
$\delta_\pi$ is obtained by taking the product of $F_{ab}$'s, over all ``horizontal'' strings of $p$, and of $\delta_{ab}$'s,
over all ``vertical'' strings of $\pi$ (note that as we only consider here pair partitions, in fact only one delta factor will
appear for any given block of the partition). Observe further that in the case of the matrix $F=F_{p,q}$ used for defining
$O^+(p,q)$, we obtain the $\delta$ numbers in the statement; it therefore suffices to establish the claim.

The proof of the claim has two steps. First, we prove that $\pi\to T_\pi$ transforms the categorical operations for partitions
into the categorical operations for the linear maps. This is indeed clear for the tensor product and for the duality. Regarding
now the composition axiom, the only problem might come from the closed circles that appear in the middle:  but here we can use
the formula $T_\cup T_\cap=(2p+q)\cdot \id$, with $2p+q=\sum_{i,j \in \Jnd_{p,q}} F_{ij}$.

Summarizing, the spaces $\tu{span}(T_\pi|\pi\in NC_2(k,l))$ form a tensor category in the sense of Woronowicz. Now since this tensor
category is generated by $T_\cap$ (that is because $NC_2$ is generated by $\cap$) the corresponding Hopf algebra is the one
obtained by using the relation $T_\cap\in \tu{Fix}(U^{\otimes 2})$. And since $T_\cap=\sum F_{ij}e_i\otimes e_j$, this algebra is
$A_o(F)$, and the proof is finished.

Note in passing, that when  $F$ is not necessarily real, the result remains the same, but with the $\delta$ numbers obtained by making the product of
$F_{ab}$'s over ``oriented'' horizontal strings, and then the product of $\delta_{ab}$'s over vertical strings.

The cases of $B^+(p,q)$ and $H^+(p,q)$  follow in a similar way, first adding the singletons and then considering all partitions
with blocks of even size.
\end{proof}

\begin{rem}
The corresponding result holds also for the classical versions of $O^+(p,q), S^+(p,q)$ and $H^+(p,q)$ (see Section
\ref{classver}), with the non-crossing partitions replaced by all partitions; the proof follows in a similar way.
\end{rem}

For $S^+(p,q)$ the above proof does not work: as already mentioned above, the problem is that, since some of the blocks have an
odd number of entries, the composition axiom does not hold for the implementation provided by the operators $T_{\pi}$ described
in \eqref{Tpi}. We therefore need to provide a new combinatorial description of the category of representations. From Theorem
\ref{thSpq} we know that $S^+(p,q)\approx H^+_p \hat{\star} S^+_q$; hence we begin by considering separately the cases $p=0$ and
$q=0$. When $p=0$, the quantum group  $S^+(0,q)$ is a quantum symmetric group with the usual fundamental representation, so the
corresponding category is $NC$, with the standard implementation providing the isomorphism (there are no `conjugate' coordinates
to deal with), as shown in \cite{Teoold} or in \cite{TeoRoland}. The situation for $q=0$ is more complicated, as the defining
fundamental  representation of $S^+(p,0)$ corresponds rather to the `sudoku' representation of $H^+_p$ than to the
representation studied from the categorical point of view in \cite{hyperoct} and \cite{TeoRoland}. We begin by defining the
relevant category of `bulleted' partitions.

\begin{deft} \label{bulletcategory}
Let (for $k, l \in \bn_0$) $P_{\bullet} (k,l)$ denote the collection of all partitions between $k$ upper and $l$ lower points,
such that each point is marked either with a black or a white circle and we identify the partitions which differ by `mirrored
markings' in some of the blocks, so that for example ${\raisebox{10pt}{\xymatrix@R=0.5pc @C=0.5pc{  & *=0{\bullet} & & *=0{\bullet}\\ *=0{\bullet}
\ar@{-}[ur] & &
*=0{\circ} \ar@{-}[ul]& *=0{\bullet} \ar@{-}[u]}}}\,$ equals  ${\raisebox{10pt}{\xymatrix@R=0.5pc @C=0.5pc{  & *=0{\circ} & & *=0{\bullet}
\\ *=0{\circ} \ar@{-}[ur] & & *=0{\bullet} \ar@{-}[ul] & *=0{\bullet} \ar@{-}[u]}}}\,$ in $P_{\bullet}(2,3)$.
 The category $P_{\bullet}=\bigcup_{k,l\in \bn_0}
P_{\bullet}(k,l)$ is defined similarly to $P$, with the result of the concatenation operation being non-zero only if the
corresponding markings match, remembering that we are allowed to replace colours of markings in any block of any given partition.

 Denote by
$NC_{\bullet}$ the subcategory of all bulleted non-crossing partitions.
\end{deft}

Given $\pi\in P_{\bullet}(k,l)$ and multi-indices $i=(i_1,\ldots,i_k)\in \Jnd_p^k$, $j=(j_1,\ldots,j_l) \in \Jnd_p^l$ define a
number $\delta_\pi(^i_j)\in\{0,1\}$ in the following way: $\delta_\pi(^i_j)=1$ if for any pair of vertices in a fixed block of
the partition $\pi$ the corresponding indices are equal if the bullets on the vertices have the same colour and are `conjugate'
if the colours of the respective bullets are different. Define further $\wt{T}_{\pi}:(\bc^{2p})^{\ot k} \to (\bc^{2p})^{\ot l}$
by the usual formula $(i_1,\ldots,i_k \in \Jnd_p)$:
\begin{equation}\label{wTpi}\wt{T}_\pi(e_{i_1}\otimes\ldots\otimes
e_{i_k})=\sum_{j_1,\ldots,j_l\in \Jnd_p}\delta_\pi(^i_j)e_{j_1}\otimes\ldots\otimes e_{j_l}.\end{equation} We are ready to formulate the
result describing the representation theory of $S^+(p,0)$.

\begin{tw} \label{catSp}
Let $U$ denote the fundamental representation of $S^+(p,0)$ introduced in Definition \ref{Spq}.  We have for all $k,l \in \bn_0$
$$\Hom(U^{\otimes k}\;U^{\otimes l})=\tu{span}(\wt{T}_\pi|\,\pi\in NC_{\bullet}(k,l)).$$
\end{tw}

\begin{proof}
It is standard to check that $NC_{\bullet}$ is a tensor category satisfying the properties needed to apply the Tannaka-Krein
duality of \cite{wor2} and that the map $\pi \to \wt{T}_{\pi}$ transforms the natural operations (concatenation with deletion of
the closed blocks, tensoring, turning upside-down) to the corresponding operations on the level of linear maps. Proposition \ref{HomSpq}
implies that the category of representations of $S^+(p,0)$ is generated by the morphisms $\wt{T}_{\pi_1}$ and $\wt{T}_{\pi_2}$
associated with the partitions
\[ \pi_1= {\raisebox{10pt}{\xymatrix@R=0.3pc @C=0.3pc{ & &  \\ & *=0{} &
\\ *=0{\bullet} \ar@{-}[ur] & & *=0{\circ} \ar@{-}[ul]}}}, \;\;\; \pi_2 = {\raisebox{10pt}{\xymatrix@R=0.3pc @C=0.3pc{  & *=0{\bullet} &
\\ & & \\ *=0{\bullet} \ar@{-}[uur] & & *=0{\circ} \ar@{-}[uul]}}} .\]
The partitions $\pi_1$ and $\pi_2$ generate the whole $NC_{\bullet}$. Indeed, it is well known that the analogous `non-bulleted'
partitions generate $NC$, and it suffices to notice that $\pi_1$ and $\pi_2$ can be first combined to obtain singletons and then
the `exchange' partition $\pi_3= {\raisebox{5pt}{\xymatrix@R=0.7pc @C=0.7pc{*=0{\bullet} \ar@{-}[d] \\
 *=0{\circ}}}}\,$ (the fact that $\wt{T}_{\pi_3}$ intertwines $U$ is in a sense a consequence of the fact
that entries of $u$ are self-adjoint). Once we know that the category generated by $\pi_1$ and $\pi_2$ contains partitions of all
shapes and we can exchange colours of markings at every vertex using $\pi_3$ the generation statement is clear.

The argument above shows that the category of representations of $S^+(p,0)$ is at least as big as $NC_{\bullet}$; so  there is an
inclusion $\QG \supset S^+(p,0)$, where $\QG$ is the compact quantum group arising as the Tannaka-Krein dual of $NC_{\bullet}$
via the described implementation. Note that the inclusion $\QG \supset S^+(p,0)$ can be described as the surjection on the level
of function algebras: $C(\QG) \twoheadrightarrow C(S^+(p,0))$. To conclude the proof, by using a standard Peter-Weyl argument, it
suffices to show that for each $k \in \bn$ the number $d_k:=\card(\Hom(1; U^{\ot k}))$ is equal to $\card(NC_{\bullet}(0,k))$.
The numbers $d_k$ are expressed by the formula
\[ d_k = h_{S^+(p,0)} \big(\Tr\, (U^{\ot k}) \big),\]
where $h_{S^+(p,0)}$ denotes the Haar state. It follows from the proof of Theorem \ref{thSpq} that when we compute the trace of
$u$ we obtain the two copies of the elements which turn up as the trace of the top-left block of $u$ in the `sudoku' picture. By
Theorem 7.2 and Corollary 7.4 of \cite{stocheasy} we know that the corresponding single random variable has (with respect to the
Haar state of $H^+_p$) the free Poisson distribution with parameter $\frac{1}{2}$. As the Haar state is obviously preserved by
the isomorphism between $S^+(p,0)$ and $H^+_p$ we have \begin{equation} d_k = 2^k \sum_{\pi \in NC(0,k)}
\left(\frac{1}{2}\right)^{\nu(\pi)} = \sum_{\pi \in NC(0,k)} 2^{k-\nu(\pi)} \label{Poissoncount}\end{equation} (where $\nu(\pi)$
denoted the number of blocks in the partition $\pi$ and we used the formula for the moments of the free Poisson distribution with
parameter $t$, see e.g.\ \cite{NicaSpeicher}). It remains to check that the number in \eqref{Poissoncount} equals $\card(NC_{\bullet}(0,k))$. But this follows from the fact
that for every non-bulleted block of a partition in $NC(0,k)$ which has $l$ legs we have $2^{l-1}$ choices of the bullet pattern.
\end{proof}

\begin{rem}
Again we have a corresponding statement for the classical version of $S^+(p,0)$, with the category $NC_{\bullet}$ replaced by
$P_{\bullet}$.
\end{rem}

Note that the categories of representations listed in Theorem \ref{catOHB} can also be described in terms of the bulleted
partitions with the fixed colouring pattern consistent with the implementation in \eqref{Tpi} (so for example if a block in a
pair partition is `vertical', both ends are given the same colour, and if it is `horizontal', then we have a colour exchange).

\subsection*{`Free product' of categories described by partitions}

Recall that for $p,q>0$ we have $S^+(p,q)= S^+(p,0)\, \hat{\star}\, S^+(0,q)$, and Theorem \ref{catSp} together with the discussion before it describe the categories of representations of both $S^+(p,0)$ and $S^+(0,q)$. Before we use this decomposition to provide a description of the category for $S^+(p,q)$, let us discuss a general free product framework.

If $\QG_1, \QG_2$ are compact quantum groups and $\QG = \QG_1\hat{\star}\, \QG_2$, although the knowledge of irreducible
representations of $\QG_1$ and $\QG_2$ suffices to describe the irreducible representations of $\QG$, the description on the
level of categories of representations seems to be quite difficult. If however respective categories for $\QG_1$ and $\QG_2$,
denoted say by $\cat_1$ and $\cat_2$, are given by (non-crossing) partitions, there is a natural candidate  for the category
corresponding to $\QG$ -- it should be given by all partitions which decompose into blocks of two colours, where the
sub-partition obtained by looking only at the blocks of the first colour belongs to $\cat_1$ and that for the second colour
belongs to $\cat_2$. The tensoring and reflection operations are defined as for usual partitions, and so is the concatenation,
with the caveat that to obtain a non-zero outcome the colours must fit together. The resulting category will be  denoted by
$\cat_1\star\, \cat_2$ (note that the construction works well also for bulleted partitions). As in the proof of Theorem
\ref{catSp}, using the Tannakian duality we can establish the existence of a compact quantum group $\QG'$, whose category of
representations equals $\cat_1\star \cat_2$, and deduce the existence of the inclusion $\QG' \supset \QG$. To show that this
inclusion is in fact an equality, we need to compare the dimensions of the fixed point spaces (or, equivalently, the moments of
the laws describing the characters of the defining fundamental representations). As the Haar state of $\QG$ was shown in
\cite{Wangfree} to be the free product of the Haar states of $\QG_1$ and $\QG_2$, the law related to $\QG$ arises as the free
convolution of the corresponding distributions for $\QG_1$ and $\QG_2$. Thus if this free convolution can be easily computed, and
the categories $\cat_1$ and $\cat_2$ are identical, the corresponding count can be performed without much difficulty. In
particular using the natural `free product' implementations $\pi \to \hat{T}_{\pi}$ defined in the spirit of \eqref{Tpi} we
obtain the following result.

\begin{tw} \label{freep}
Let $n,m \in \bn$. The category corresponding to the standard fundamental representations of $S^+_n \hat{\star}\, S^+_m$ (respectively, $H^+_n
\hat{\star}\, H^+_m$, $B^+_n \hat{\star}\, B^+_m$  and $O^+_n \hat{\star}\, O^+_m$) is equal to the span of $\{\hat{T}_{\pi}: \pi
\in \mathcal{C}\star \mathcal{C}\}$, where $\mathcal{C}= NC$ (respectively, $NC_{12}$, $NC_{\tu{even}}$ and $NC_2$).
\end{tw}

\begin{proof}
The proof follows as in Theorem \ref{catSp}, so we only provide the argument for the equality of the cardinality of partitions in
$(\mathcal{C}\star\mathcal{C})(0,k)$ and the dimension of the fixed point spaces for tensor powers of the defining
fundamental representations. Fix $x\in\{.,2,12,\tu{even}\}$ and for each $t>0$ consider the probability measure $\mu_t$ having as
moments $\int \lambda^k\,d\mu_t({\lambda})=\sum_{\pi\in NC_x}t^{\nu(\pi)}$, where $\nu(\pi)$ denotes the number of blocks in
$\pi$. The results in \cite{TeoRoland} imply that $\mu_1$ is the measure describing the distribution of the character of the
defining representation of the quantum group corresponding to $x$, and also that the measures $\{\mu_t: t >0\}$ form a semigroup
with respect to the free convolution $\boxplus$ (see \cite{voic}). A standard argument using the definition of the free product
construction implies that the character of the defining representation for $\QG_n^+\hat{\star}\,\QG_m^+$ follows the law
$\mu_1\boxplus\mu_1=\mu_2$, with moments given by $\int \lambda^k\,d\mu_2(\lambda)=\sum_{\pi\in NC_x}2^{\nu(\pi)}$. But this is
exactly the number of partitions featuring in $(\mathcal{C}\star\mathcal{C})(0,k)$  -- using the surjective map
$\mathcal{C}\star\mathcal{C}\to \mathcal{C}$ we see that each element in the preimage of a given $\pi\in \mathcal{C}(0,k)$ is the
partition $\pi$ with each block given one of the two colours; thus the preimage has $2^{\nu(\pi)}$ elements.
\end{proof}

Before we prove an analogous, more technical,  result for the quantum group $S^+(p,q)$ isomorphic to $S^+(p,0)\, \hat{\star}\, S^+(0,q)$ we need a
simple observation related to the free cumulants of probability measures. All relevant definitions can be found in  Lectures 11
and 12 in \cite{NicaSpeicher}.

\begin{lem}\label{lemcum}
Assume that $\nu$ is a compactly supported probability measure on $\mathbb{R}$. Let $X$ be a real-valued random variable with
distribution $\nu$, let $\mu=\nu \boxplus \nu$ and let $\nu'$ be the law of $2X$. Denote the respective free cumulants of $\mu$
and of $\nu'$ by $(\kappa_k(\mu))_{k=1}^{\infty}$ and $(\kappa_k(\nu'))_{k=1}^{\infty}$. Then for each $k \in \bn$
\[ \kappa_k(\nu') = 2^{k-1} \kappa_k(\mu).\]
\end{lem}

\begin{proof}
It follows from Proposition 12.3 in \cite{NicaSpeicher} that $\kappa_n(\nu) = \frac{1}{2} \kappa_n(\mu)$. Further the
 fact that the moments and cumulants are related by the Speicher's moment-cumulant formula (\cite{Rolandmomcum}):
\begin{equation} \label{momcum} d_k = \sum_{\pi \in NC(0,k)}\, \prod_{b: \tu{block in } \pi}\, \kappa_{\tu{card}(b)},  \end{equation}
 implies that $\kappa_k(\nu') = 2^k \kappa_k(\nu)$.
\end{proof}

The following theorem completes the description of the categories for the main family of quantum groups studied in this paper. 

\begin{tw} \label{catfree}
Let $u$ denote the fundamental representation of $S^+(p,q)$ introduced in Definition \ref{Spq}.  Then for all $k,l \in \bn_0$
$$\Hom(U^{\otimes k};U^{\otimes l})=\tu{span}(\wt{T}_\pi|\pi\in (NC_{\bullet} \star NC)(k,l)).$$
\end{tw}
\begin{proof}
As in the proof of Theorem \ref{freep} we only provide the argument for the equality of the number of partitions in
$(NC_{\bullet} \star NC)(0,k)$ and the $k$-th moment of the character of the defining representation of $S^+(p,q)$. The latter
can be expressed via the moment-cumulant formula \eqref{momcum}, with the free cumulants $(\kappa_k)_{k=1}^{\infty}$ given by the
sum of the free cumulants of the corresponding laws of the characters for $S^+(p,0)$ and $S^+(0,q)$ (this is a consequence of the
discussion before Theorem \ref{freep} and Proposition 12.3 in \cite{NicaSpeicher}). The free cumulants for the second law (which
is the free Poisson distribution) are equal to $1$. As the free Poisson laws form a free convolution semigroup,
the proof of Theorem \ref{catSp} and Lemma \ref{lemcum} imply that the free cumulants for the first law are equal to $2^{k-1}$.  Hence the
moment-cumulant formula for the $k$-th moment of the law we are interested in yields:
\[ d_k = \sum_{\pi \in NC(0,k)}\,  \prod_{b: \tu{block in } \pi}\, (2^{\tu{card}(b)-1} +1).\]
It remains to note that the above number indeed corresponds to the number of partitions in $(NC_{\bullet} \star NC)(0,k)$: for
each block $b$ in a given non-crossing partition $\pi$ of $k$-points we can choose whether it `comes' from $NC_{\bullet}$ or $NC$
and in the first case we have additionally $2^{\tu{card}(b)-1}$ choices of inequivalent bulleting patterns.
\end{proof}

\section{Relations between the two-parameter families}

Recall the diagram \eqref{diag1} describing the relations between $O^+_n$, $B^+_n$, $H^+_n$ and $S^+_n$. Universal properties
imply that the corresponding diagram can be drawn in the two-parameter case; for each $p, q \in \bn_0$ we have

\begin{equation} \label{diag2}\begin{matrix}
A_o(p,q)&\to&A_b(p,q)\\
\\
\downarrow&&\downarrow\\
\\
A_h(p,q)&\to&A_s(p,q)
\end{matrix}
\;\;\;\;\;\;\;\;\;\;\;\;\;\;\;\;\;\;\;\;
\begin{matrix}
O^+(p,q)&\supset&B^+(p,q)\\
\\
\cup&&\cup\\
\\
H^+(p,q)&\supset&S^+(p,q)
\end{matrix}\end{equation}

Below we describe further connections between the quantum groups studied above, first discussing the general framework.

\begin{deft}\label{4groups}
Assume that we have surjective morphisms of Hopf $C^*$-algebras, corresponding to inclusions of compact quantum groups, as
follows: \begin{equation} \label{4diag}\begin{matrix}
A&\to&B\\
\\
\downarrow&&\downarrow\\
\\
C&\to&D
\end{matrix}
\quad\ \quad :\quad\ \quad
\begin{matrix}
\QG&\supset&\QH\\
\\
\cup&&\cup\\
\\
\QK&\supset&\QL
\end{matrix}\end{equation}
\begin{enumerate}
\item We write $\QG=\,<\QH,\QK>$ if the $*$-algebra ideal $\ker(A\to B)\cap\ker(A \to C)$ contains no nonzero Hopf ideal.
\item We write $\QL=\QH\cap \QK$ when $\ker(A\to D)$ is the Hopf ideal generated by $\ker(A\to B)$ and $\ker(A\to C)$.
\end{enumerate}
\end{deft}

The definitions above can be easily seen to coincide with the standard ones in the classical case.

More generally, given morphisms $\alpha: A \to B$ and $\beta:A \to C$ as above, one can define a quantum group $<\QG,\QH>\,\subset
\QK$ by the formula $C(<\QG,\QH>)=A/I$, where $I$ is the biggest Hopf ideal contained in $\ker(\alpha)\cap\ker(\beta)$. Once
again, this definition agrees with the usual one in the classical case.

These notions are best understood in terms of the Hopf image formalism, introduced in \cite{bb2}. Consider indeed the morphism
$(\alpha,\beta):C\to A\times B$. Then $C(<\QG,\QH>)$ is the Hopf image of $(\alpha,\beta)$, and we have $<\QG,\QH>\,=\QK$ if and
only if $(\alpha,\beta)$ is inner faithful.

Similarly given a 4-term diagram as in the above definition, we can form the Hopf ideal $J=\,<\ker(A\to B),\ker(A\to C)>$, and
define a quantum group $\QH\cap \QK$ by $C(\QH\cap \QK)=A/J$. We have $\QL=\QH\cap \QK$ if and only if the above condition (2) is
satisfied.

In order to deal effectively with part (1) of Definition \ref{4groups}, we use the following Tannakian reformulation.

\begin{lem} \label{4grreformulated}
Consider the diagram \eqref{4diag} and denote the Hopf morphism between $A$ and $B$ (respectively, between $A$ and $C$) by
$\alpha$(respectively, $\beta$). We have $\QK=\,<\QG,\QH>$ if and only if
$$\tu{Fix}(R)=\tu{Fix}((\id\otimes\alpha)R)\cap \tu{Fix}((\id\otimes\beta)R)$$
for any representation $R$ arising as a tensor product between $U$ and $\bar{U}$'s, with $U$ being the fundamental representation of $\QG$.
\end{lem}

\begin{proof}
The result is a consequence of general considerations in \cite{bb2}; we sketch the proof below.

Let $I$ be the biggest Hopf ideal contained in $\ker(\alpha)\cap\ker(\beta)$, whose existence follows from \cite{bb2}, and let
$\QK_1=\,<\QG,\QH>$ be the compact quantum group given by $C(\QK_1)=C(\QK)/I$.

By Frobenius duality, the collection of equalities in the statement is equivalent to the following collection of equalities:
$$\Hom(R;S)=\Hom((\id\otimes\alpha)R;(\id\otimes\alpha)S)\cap \Hom((\id\otimes\beta)R;(\id\otimes\beta)S),$$
where now both $R$ and $S$ are representations arising as tensor products of several copies of $U$ and $\bar{U}$.

According to the general results of Woronowicz in \cite{wor2}, the Hom spaces on the right form a Tannakian category, which
should therefore correspond to a certain compact quantum group $\QK_2$. Our claim is that we have $\QK_1=\QK_2$. Indeed, this
follows from Theorem 8.4 of \cite{bb2} applied to the morphism $(\alpha, \beta):C \to A \times B$ (and can be viewed as a consequence of the general Peter-Weyl type results in \cite{wor1} and the Tannakian duality results in \cite{wor2}).

In particular $\QK=\QK_1$ is equivalent to $\QK=\QK_2$, and we are done.
\end{proof}


Before we formulate the specific intersection/generation results in the cases in which we are interested, we need another lemma:

\begin{lem}\label{intersect}
Let $\alg$ be a $C^*$-algebra, $n \in \bn$ and let $U\in M_{n}(\alg)$ be a unitary matrix whose entries are partial isometries. If for each $z \in \{1, \ldots,n\}$ there is $\sum_{y=1}^n U_{y,z}=1$, then each $U_{y,z}$ is an orthogonal projection.
\end{lem}

\begin{proof}
Fix $y,z\in \{1, \ldots,n\}$. We have $U_{y,z} = 1 - \sum_{x\neq y} U_{x,z}$
Proposition \ref{ort0} implies that multiplying the last equality by $U_{y,z}^*$ on the right yields
$U_{y,z} U_{y,z}^*  = U_{y,z}^*$. Hence $U_{y,z}$ is a selfadjoint projection.
\end{proof}

\begin{tw}
The following relations hold:
\begin{enumerate}
\item $O^+(p,q)=\,<B^+(p,q),H^+(p,q)>$.
\item $S^+(p,q)=B^+(p,q)\cap H^+(p,q)$.
\end{enumerate}
\end{tw}

\begin{proof}
(1) follows from Lemma \ref{4grreformulated}, Theorem \ref{catOHB}, and the equality  $NC_2=NC_{12}\cap NC_{\tu{even}}$.

(2) is a direct consequence of Lemma \ref{intersect}.
\end{proof}

\subsection*{Free products}

We recall from Section \ref{OBS} that we have a canonical isomorphism $A_o(p,q)\simeq A_o(2p+q)$. It is well known that for any
$n,m \in \bn$ there exists a natural surjective map $A_o(m+n)\to A_o(m) \star A_o(n)$. Similar maps turn out to exist in all the cases
considered in this paper.

\begin{prop}\label{freeone}
Let $g \in \{o,b,h,s\}$. There exists a natural surjective map $A_g(p,q)\to A_g(p,0)\star A_g(0,q)$  such that the following diagram
(in which the vertical maps are the ones introduced earlier in this section) is commutative:
$$\begin{matrix}
A_o(p,q)&\to&A_o(p,0)\star A_o(0,q)\\
\\
\downarrow&&\downarrow\\
\\
A_g(p,q)&\to&A_g(p,0)\star A_g(0,q)
\end{matrix}$$
The horizontal maps in the diagram above intertwine respective comultiplications.
\end{prop}

\begin{proof}
Let $U_1 \in M_{2p}(A_g(p,0))$ and $U_2 \in M_{q}(A_g(0,q))$ be fundamental unitary representations of the respective compact
quantum groups. Consider the unitary matrix $\begin{bmatrix} U_1  & 0 \\ 0 & U_2 \end{bmatrix}$ viewed as an element of $M_{2p+q}
(A_g(p,0)\star A_g(0,q))$. It is easy to see that it satisfies all the conditions required of the generating matrix in $M_{2p+q}
(A_g(p,q))$, so the universality implies the existence of a $^*$-homomorphism requested by the proposition. It is easily seen to
be surjective, as entries of $U_1$ (respectively, $U_2$) generate $A_g(p,0)$ (respectively, $A_g(0,q)$). An explicit description
of the comultiplications involved implies the last statement of the proposition.

Recall that the existence of surjective (compact quantum group) morphisms from $A_o(p,q)$ to $A_g(p,q)$ followed in an analogous
way; the fact that the diagram above is commutative is thus a direct consequence of the definitions of the maps considered.
\end{proof}

Using directly the language of compact quantum groups we obtain the following corollary:

\begin{cor}\label{freecor}
Let $\QG \in \{O^+, H^+, B^+, S^+\}$. There exist natural inclusions
$$\begin{matrix}
O^+(p,q)&\supset&O^+(p,0)\,\hat{\star}\,O^+(0,q)\\
\\
\cup&&\cup\\
\\
\QG(p,q)&\supset&\QG(p,0)\,\hat{\star}\, \QG(0,q)
\end{matrix}$$
\end{cor}

In the case of $A_o(p,q)$ the map described  in Proposition \ref{freeone} and the isomorphisms/surjections recalled before it can
be combined into the following commutative diagram:
$$\begin{matrix}
A_o(p,q)&\to&A_o(p,0)\star A_o(0,q)\\
\\
\downarrow&&\downarrow\\
\\
A_o(2p+q)&\to&A_o(2p)\star A_o(q)
\end{matrix}$$
(the commutativity is the consequence of the fact how the isomorphism was described in the proof of Theorem \ref{thmOpq} and the
block-diagonal form of the matrices $C_{p,q}$ implementing it).

Theorem \ref{thSpq} implies that the inclusion $S^+(p,q)\supset S^+(p,0)\,\hat{\star}\,S^+(0,q)$ is in fact an isomorphism. In all the
other cases when both $p$ and $q$ are non-zero (for $B^+(p,q)$ we actually need to assume $q>1$ to avoid trivialities)
$\QG(p,0)\,\hat{\star}\,\QG(0,q)$ is a proper quantum subgroup of $\QG(p,q)$. This is the content of the next proposition.

\begin{prop}
If $p,q \neq 0$ then the inclusions $O^+(p,0)\,\hat{\star}\,O^+(0,q) \subset O^+(p,q)$ and  $H^+(p,0)\,\hat{\star}\,H^+(0,q)
\subset H^+(p,q)$ are proper. If $p>0$ and $q>1$ then the inclusion $B^+(p,0)\,\hat{\star}\,B^+(0,q) \subset B^+(p,q)$ is proper.
\end{prop}

\begin{proof}
It is enough to show that the corresponding homomorphisms $A_o(p,q)\to A_o(p,0)\star A_o(0,q)$, $A_h(p,q)\to A_h(p,0)\star
A_h(0,q)$ and $A_b(p,q)\to A_b(p,0)\star A_b(0,q)$ are not injective; in other words it suffices to find unitary matrices
satisfying the defining conditions for $A_o(p,q)$, $A_h(p,q)$ and $A_b(p,q)$ which are not of the form $\begin{bmatrix} U_1  & 0
\\ 0 & U_2
\end{bmatrix}$ with $U_1$ a $2p$ by $2p$ matrix. In the case of the orthogonal and the bistochastic group the existence of such matrices is visible
already at the commutative level. 

Consider then the case of $H^+$. It suffices to find suitable matrices for $p=q=1$. Let then $Q_A, Q_B, P_A, P_B$ be non-zero
orthogonal projections on some Hilbert space $\Hil$ such that $Q_A + Q_B+ P_A+ P_B=1$ and let $A,B, C, D$ be partial isometries
in $B(\Hil)$ such that $AA^*=Q_A$, $A^*A=P_A$, $BB^*=Q_B$, $B^*B=P_B$, $CC^*=1-C^*C = P_B +Q_A$, $DD^*=1-D^*D=P_A +Q_B$. The
existence of such partial isometries can be assured by choosing all the projections to have infinite-dimensional ranges. Then the
matrix
\[ \begin{bmatrix} A & B & C \\ B^* & A^* & C^* \\ D &D^* & 0\end{bmatrix}\]
can be checked to be a unitary satisfying assumptions listed in Definition \ref{Hpq}.
\end{proof}

Note that the above construction excludes the possibility of the partial isometries $A,B,C,D$ generating a commutative algebra.
This is a general fact -- in Theorem \ref{classHpq} below we will see that if we additionally request commutativity then the
`off-diagonal' terms of a unitary satisfying the defining conditions of $A_h(p,q)$ must vanish.


\subsection*{Quantum permutation groups as quotient quantum groups of $H^+(p,q)$}

\begin{deft}
If $\QG_1$ and $\QG_2$ are compact quantum groups and $\pi:C(\QG_1) \to C(\QG_2)$ is a unital injective $^*$-homomorphism
intertwining the respective coproducts, then $\QG_1$ is said to be a quotient quantum group of $\QG_2$ (alternatively, $\QG_2$ is
said to be a quantum group extension of $\QG_1)$.
\end{deft}

 Consider once again the quantum group $H^+(p,q)$. As the generators of $A_h(p,q)$, denoted by $U_{z,y}$ ($z,y \in \Jnd_{p,q}$)
are partial isometries, the operators $P_{z,y}:=U^*_{z,y} U_{z,y}$ are projections.

\begin{tw}
The $C^*$-subalgebra
of $A_h(p,q)$ generated by $P_{z,y}$ ($z,y \in \Jnd_{p,q}$) is isomorphic to $A_s(2p+q)$.
\end{tw}

\begin{proof}
Let $A$ denote the $C^*$-subalgebra
of $A_h(p,q)$ generated by $P_{z,y}$ ($z,y \in \Jnd_{p,q}$).
Note first that the generators $\{P_{z,y}:z,y \in \Jnd_{p,q}\}$ satisfy the same relations as generators of $A_s (2p+q)$; indeed
for $z,y \in \Jnd_{p,q}$
\begin{align*} \sum_{x \in \Jnd_{p,q}} P_{z,x} P_{y,x} &= \sum_{x \in \Jnd_{p,q}} U^*_{z,x} U_{z,x} U^*_{y,x} U_{y,x} =
\delta_{z}^y \sum_{x \in \Jnd_{p,q}} U^*_{z,x} U_{z,x} U^*_{z,x} U_{z,x} \\&= \delta_{z}^y \sum_{x \in \Jnd_{p,q}} U^*_{z,x}
U_{z,x} =  \delta_{z}^y \sum_{\bar{x} \in \Jnd_{p,q}} U_{\bar{z},\bar{x}} U_{\bar{z},\bar{x}}^*  = \delta_{z}^y 1.
 \end{align*}
Similarly it follows from Corollary  \ref{corort} that
\begin{align*} \sum_{x \in \Jnd_{p,q}} P_{x,z} P_{x,y} &= \sum_{x \in \Jnd_{p,q}} U^*_{x,z} U_{x,z} U^*_{x,y} U_{x,y} =
\delta_{z}^y \sum_{x \in \Jnd_{p,q}} U^*_{x,z} U_{x,z} U^*_{x,z} U_{x,z} \\&= \delta_{z}^y \sum_{x \in \Jnd_{p,q}} U^*_{x,z}
U_{x,z} = \delta_{z}^y 1.
 \end{align*}
If $\Com$ denotes the coproduct of $\alg_h(p,q)$ then for $z,y \in \Jnd_{p,q}$
\[ \Com(P_{z,y}) = \sum_{x \in \Jnd_{p,q}} P_{z,x} \ot P_{x,y}.\]
Hence $A$ is isomorphic to $C(\QG)$, where $\QG$ is a quantum subgroup of $S^+_{2p+q}$. It remains to check that $\QG$ is
actually $S^+_{2p+q}$. For that it suffices to analyse the corresponding categories of representations, and even more
specifically to check that cardinalities of fixed point sets of tensor powers of the fundamental representations of $\QG$
coincide with those appearing for $S^+_{2p+q}$ (see Section \ref{coreps-partitions}, where this method was applied several times).

In the partition picture described in Section \ref{coreps-partitions} the passage from $H^+(p,q)$ to $\QG$ corresponds to
considering only partitions with even number of upper and lower points and composing them (both from above and from below) with a
suitable number of copies of the partition  ${\raisebox{8pt}{\xymatrix@R=0.4pc @C=0.4pc{   *=0{} \ar@{-}[d] &  *=0{} \ar@{-}[d]  \\
*=0{} \ar@{-}[r] &  *=0{} \\  *=0{} \ar@{-}[u] &  *=0{} \ar@{-}[u]}}}\,\,$. This means that we are left only with these
partitions in $NC(2k,2l)$ in which the  pairs of points beginning at an odd place are necessarily joined (it is easy to see that
all such partitions arise in this procedure); we denote them by $NC_{\tu{join}}(2k,2l)$. Now collapsing the pairs listed above we
obtain a bijection between $NC_{\tu{join}}(2k,2l)$ and $NC(k,l)$. This suffices to perform the count needed to finish the proof.
\end{proof}

Using the dual language and noticing that the map $A_p(p,q) \to A$ is an injective morphism preserving the respective
coproducts, we obtain the following corollary.

\begin{cor}
The quantum group $H^+(p,q)$ is an extension of $S^+_{2p+q}$.
\end{cor}

\section{Classical versions and interpretations in terms of classical/quantum symmetries} \label{classver}

Each of the quantum groups studied above has a classical version. These are understood as follows: if $\QG$ is a compact quantum
group then the quotient of the algebra $C(\QG)$ by its commutator ideal is isomorphic to the algebra of continuous functions on a
certain uniquely determined compact group $G$. Then we call $G$ the classical version of $\QG$. Note that if we use the notation
$G(p,q)$ for the classical version of the quantum group $G^+(p,q)$, then Corollary \ref{freecor} and a straightforward analysis
of classical versions of free products described in Section 1 implies that we have the following inclusions:
$$\begin{matrix}
O(p,q)&\supset&O(p,0)\times O(0,q)\\
\\
\cup&&\cup\\
\\
G(p,q)&\supset&G(p,0)\times G(0,q)
\end{matrix}$$

The combination of results from the previous section and \cite{TeoRoland} allow us to identify the classical version of
$O^+(p,q)$ with $O_{2p+q}$, the classical version of $B^+(p,q)$ with $O_{2p+q-i(p,q)}$ (where $i(p,q)=1$ if $pq=0$ and $i(p,q)=2$
if $p,q>0$) and the classical version of $S^+(p,q)$ with $H_p \times S_q$.

\begin{tw} \label{classHpq}
The classical version of $H^+(p,q)$ is $(\bt^p \rtimes H_p) \times H_q$ (recall that $\bt^p \rtimes H_p$ is the usual isometry
group of $\bt^p$).
\end{tw}

\begin{proof}
Observe first that if $\{U_{z,y}:z,y \in \Jnd_{p,q}\}$ are commuting partial isometries which satisfy the condition \eqref{ad5}
and form a unitary matrix, then $U_{i\alpha, M}= U_{M,i\alpha}=0$ for all $i\alpha \in \Jnd_p$ and $M \in \{1, \ldots, q\}$.
Indeed, the projections $U_{i\alpha, M}^* U_{i\alpha, M}$ and $U_{\ib\alpha, M}^* U_{\ib\alpha, M} = U_{i\alpha, M} U_{i\alpha,
M}^*$ are then equal; as due to Proposition \ref{ort0} they are pairwise orthogonal to each other, they must vanish (similar
argument applies to $U_{M, i\alpha}$). Hence to prove the theorem it suffices to consider separately cases $p=0$ and $q=0$. The
first one was treated in \cite{hyperoct}.

Assume then that $q=0$. The result can be deduced from calculations in \cite{Jyot}, but we can also offer an explicit isomorphism.
Denote the images of elements of $\{U_{i\alpha, j\beta}:i\alpha, j \beta \in \Jnd_{p}\}$ in the quotient space (with respect to
the commutator ideal of $\alg_h(p,0)$  by $\{\hat{U}_{i\alpha, j\beta}:i\alpha, j \beta \in \Jnd_{p}\}$ and let the quotient
$C^*$-algebra be denoted by $\alg_h^{\tu{com}} (p,0)$. Identify
 $C( \bt^p \rtimes H_p)$ as a $C^*$-algebra in a natural way with $C(\bt^p) \ot C(S_p) \ot C(\underbrace{\bz_2 \times \cdots \times
\bz_2}_{p \tu{ times}})$. Let  $\chi_{\alpha, \beta}\in C(S_p)$ be the characteristic function of the set of these permutations
which map $\alpha$ into $\beta$ and let $\kappa_{0, \alpha}, \kappa_{1, \alpha} \in C(\bz_2 \times \cdots \times \bz_2)$ be
characteristic functions of sets which have respectively $0$ or $1$ in the $\alpha$-coordinate. One can then check that the map
\[ \hat{U}_{0\alpha, 0 \beta}  \to z_{\alpha} \ot \chi_{\alpha, \beta} \ot \kappa_{0, \alpha},\] and
\[\hat{U}_{0\alpha, 1 \beta} \to  z_{\alpha} \ot \chi_{\alpha, \beta} \ot \kappa_{1, \alpha}\]
extends to an isomorphism of $\alg_h^{\tu{com}} (p,0)$ with $C(\bt^p \rtimes H_p)$; this isomorphism preserves also the respective
coproducts.
\end{proof}

Note that from the point of view of the philosophy presented in \cite{TeoRoland} $H^+(p,0)$ can therefore be viewed as the
liberation of a group of isometries of $p$-copies of the circle. This is consistent with the results of Section 5 of
\cite{groupdual}, where $H^+(p,0)$ was first discovered  as the quantum isometry group of $C(\widehat{\mathbb{F}_p})$ - in other
words, of the `free torus' (note that $C(\widehat{\mathbb{F}_p})$ is the universal $C^*$-algebra generated by $p$ unitaries, whereas
$C(\mathbb{T}^p)$ is the universal $C^*$-algebra generated by $p$ commuting unitaries). Similarly the quantum group $O^+_p$ can be
viewed as the quantum symmetry group of $\sph^{2p}$, the free sphere studied in \cite{TeoDeb}. To strengthen this analogy we
observe the following fact.

\begin{prop}
There is a natural surjection from $C(\sph^{2p})$ to $C(\widehat{\mathbb{F}_p})$.
\end{prop}

\begin{proof}
Recall that $C(\sph^{2p})$ is the universal $C^*$-algebra generated by selfadjoint operators $\{x_i,y_i:i=1,\ldots,p\}$ such that
$\sum_{i=1}^p x_i^2 + y_i^2 = 1$ and $C^*(\mathbb{F}_p)\approx C(\widehat{\mathbb{F}_p})$  is the universal  $C^*$-algebra
generated by unitaries $\{u_i:i=1,\ldots,p\}$. It suffices to observe that the quotient of $C(\sph^{2p})$ by the closed two-sided
ideal generated by expressions $x_i y_i=y_i x_i, x_i^2+y_i^2 = \frac{1}{p}$ ($i=1,\ldots,p$) is naturally isomorphic to
$C(\widehat{\mathbb{F}_p})$ -- the isomorphism maps images of $x_i$ and $y_i$ in the quotient respectively to
$\frac{1}{2\sqrt{p}} (u_i + u_i^*)$ and $\frac{i}{2\sqrt{p}} (u_i - u_i^*)$.
\end{proof}

\section{Quantum group $H_s^+(p,q)$}

It is natural to ask whether the two-parameter families studied in this paper in a sense exhaust all possible natural choices in
the `$F_{p,q}$' framework. When we studied $A_h(p,q)$ and $A_s(p,q)$ we required the entries of the fundamental corepresentation
$U$ to satisfy respectively the conditions of the type $x=x x^* x$ and $x=x^*x$, which have natural descriptions in terms of the
intertwiners between $U$ and its tensor powers (see Proposition \ref{HomSpq} and \ref{HomHpq}). We could equally well describe in
terms of the intertwiners conditions on the entries of the type say $x=xx^*xx^*$. These introduce only one new possibility, as
explained by the next proposition.

\begin{prop}
Let $x$ be a bounded operator on a Hilbert space, $k\in \bn$. If $x= (x^*x)^k$ or $x = (xx^*)^k$ then $x$ is an orthogonal
projection. If $x=x(x^*x)^k$, then $x$ is a partial isometry. If $x=x^* (xx^*)^k$, then $x$ is a partial symmetry (a selfadjoint
partial isometry).
\end{prop}

\begin{proof}
We only prove the last statement, leaving the first two to the reader. Suppose $x=x^* (xx^*)^k$. Then $x^*=x (x^*x)^k$, so   $x^2
= (xx^*)^{k+1}$ and $(x^*)^2 = (x^*x)^{k+1}$. Taking the adjoint of the last relation we obtain that  $(x^*x)^{k+1} =
(xx^*)^{k+1}$. As both $xx^*$ and $x^*x$ are positive, the last equality holds if and only if $x$ is normal. Now it remains to
observe that for $z \in \bc$ the equality $z= \bar{z} |z|^k$ holds if and only if $z\in\{-1,0,1\}$ and apply the spectral theorem
to end the proof.
\end{proof}

The above proposition motivates the next definition.

\begin{deft}\label{Hspq}
The algebra $A_{hs}(p,q)$ is the universal $C^*$-algebra generated by partial symmetries $\{U_{z,y}:z,y \in \Jnd_{p+q}\}$ which satisfy
all the relations required of generators of $A_o(p,q)$.
\end{deft}

\begin{prop}\label{directHspq}
The algebra $A_{hs}(p,q)$ is the universal $C^*$-algebra generated by two families of partial symmetries $\{U_{i\alpha,j\beta}:i\alpha,
j\beta \in\Jnd_p\}$ and $\{U_{M,N};M,N \in \{1, \ldots,q\}\}$, such that both matrices $(U_{i\alpha,j\beta})_{i\alpha, j\beta
\in\Jnd_p}$ and $(U_{M,N})_{M,N=1}^q$ are magic unitaries and moreover
\begin{equation} U_{i \alpha, j \beta} = U_{\bar{i}
\alpha, \bar{j} \beta}, \;\;\; i\alpha, j\beta \in \Jnd_p. \label{adjointhsp}\end{equation}
\end{prop}

\begin{proof}
It suffices to show that whenever $i\alpha \in \Jnd_p$ and $M \in \{1, \ldots,q\}$ then $U_{i \alpha, M} = 0 = U_{M, i\alpha}$.
Due to Proposition \ref{ort0} we have  $(U_{i \alpha, M})^2 =0 = (U_{M, i\alpha})^2= 0$, and the proof is finished.
\end{proof}

Before we formulate the main result we need to recall the definition of a `higher-order' quantum hyperoctahedral
group first defined in \cite{hyperoct} and later studied for example in \cite{stocheasy} and \cite{Bessel}.

\begin{deft} \label{H4pq}
Denote by  $A_h^{(4)}(n)$ the universal $C^*$-algebra generated by the entries of an $n$ by $n$ unitary $U$, such that $\bar{U}$
is also unitary, for each $i,j=1,\ldots, n$ the entry $U_{ij}$ is normal, and $U_{ij}^4 =U_{ij} U_{ij}^*$ is an orthogonal
projection. When $U \in M_n \ot A_s(n)$ is interpreted as the fundamental unitary corepresentation, we  view $A_h^{(4)}(n)$ as the
algebra of continuous functions on the quantum hyperoctahedral group of order 4 on $n$ coordinates,  $H^{4+}_n$.
\end{deft}

Note that if in the above definition $4$ is replaced by $2$ we obtain the quantum hyperoctahedral group $H^+_n$ introduced in
Definition \ref{unihyp}. The factors 2 and 4 can be replaced by arbitrary $s \in \bn$, as shown in the papers cited above.

\begin{tw} \label{thHspq}
The algebra $A_{hs}(p,q)$ is the algebra of continuous functions on a compact quantum group, denoted further $H_s^+(p,q)$. The unitary
$U=(U_{z,y})_{ z,y \in \Jnd_{p,q}} \in M_{2p+q} \ot A_{hs}(p,q)$ is the fundamental representation of $S^+(p,q)$. The algebra
$A_{hs}(p,q)$ is isomorphic to the free product $A_{hs}(p,0) \star A_h(q)$; on the level of quantum groups we have
$H_s^+(p,q)\approx H^{4+}_p \, \hat{\star} \, H^+_q$.
\end{tw}

\begin{proof}
By Proposition \ref{directHspq} it suffices to consider separately the cases $p=0$ and $q=0$. The fact that $A_{hs}(0,q) \approx
A_h(q)$ (with the isomorphism preserving natural fundamental corepresentations) is immediate. Assume then that $q=0$. Note that
the fundamental corepresentation has then the form
\begin{equation} \label{fundhs} \begin{bmatrix} p_{A} - q_A & p_B - q_B & p_C - q_C &p_D - q_D & \cdots\\
 p_{B} - q_B & p_A - q_A & p_D - q_D &p_C - q_C & \cdots \\ \vdots & \vdots &\vdots & \vdots & \ddots \end{bmatrix},\end{equation}
where for each $X \in \{A,B,C,D,\ldots\}$ the pair $(p_X,q_X)$ is a pair of mutually orthogonal projections.

Following the permutation procedure described in the proof of Theorem \ref{thSpq} we can describe the fundamental representation
as a `sudoku' unitary, whose entries are partial symmetries; and further considering separately positive and negative part of
each partial symmetry express the relations between them by placing them in a `double sudoku', i.e.\ a magic unitary of dimension
$4p$ by $4p$, which has the following block-matrix structure:
\begin{equation} \label{fundhs2} \begin{bmatrix} P & Q& R &S \\ Q & P &S &R \\ R&S & P&Q \\ S&R & Q& P\end{bmatrix},\end{equation}
with $P,Q,R,S$ being $p$ by $p$ matrices whose entries are projections. Note that this fits with the fact that $H^+_s(p,0)$ is a
quantum subgroup of $H^+_{2p}$. We can now apply Theorem 2.3 (2) of \cite{TeoRolandV} to deduce that the algebra $A_{hs}(p,0)$ is
isomorphic to $C(H^{4+}_p)$; one can check that this isomorphism is also a Hopf $^*$-algebra morphism.
\end{proof}

Note that the fundamental corepresentation of $C(H^{4+}_p)$ in \eqref{fundhs} is in a sense intermediate between that of
\cite{Bessel} (which has dimension $p$ and is given by the matrix with entries of the type $p_A - q_A + i (p_B - q_B)$) and that
given in \eqref{fundhs2} considered in \cite{TeoRolandV} (which has dimension $4p$). Once again we describe the quantum group
studied in terms of the intertwiners between the tensor powers of the fundamental corepresentation.

\begin{prop} \label{HomHspq}
The algebra $A_{hs}(p,q)$ is the universal $C^*$-algebra generated by the entries of a unitary $2p+q$ by $2p+q$ matrix $U$ such that the
vector $\xi$ defined in Proposition \ref{HomOpq} is a fixed vector for $U^{\ot 2}$ and the map $e_{i\alpha} \to e_{\bar{i}\alpha} \ot
e_{i\alpha}\ot e_{\bar{i}\alpha}$, $e_M \to e_{M} \ot e_{M} \ot e_M$ defines a morphism in $\Hom (U; U^{\ot 3})$.
\end{prop}

\begin{proof}
Identical to that of Proposition \ref{HomHpq}.
\end{proof}

It remains to describe the category describing the representations of $H_s^+(p,q)$. Once again we will first separately consider
the case $p=0$ (when everything reduces to the usual computations with $H^+_q$ and the corresponding category is
$NC_{\tu{even}}$) and the case $q=0$. For the second case recall Definition \ref{bulletcategory} and denote by $NC_{\bullet,
\tu{even}}$ the subcategory given by the bulleted partitions of even size.

\begin{tw} \label{catHsp}
Let $U$ denote the fundamental representation of $H^+_s(p,0)$ introduced in Definition \ref{Hspq}.  We have
$$\Hom(U^{\otimes k};U^{\otimes l})=\tu{span}(\wt{T}_\pi|\,\pi\in NC_{\bullet, \tu{even}}(k,l)),$$
where the implementing maps are defined as in \eqref{wTpi}.
\end{tw}

\begin{proof} The proof is very similar to that of Theorem \ref{catSp}.
It is standard to check that $NC_{\bullet, \tu{even}}$ is a tensor category satisfying the properties needed to apply the Tannaka-Krein
duality of \cite{wor2} and that the map $\pi \to \wt{T}_{\pi}$ transforms the natural operations (concatenation, tensoring, adjoint)
to the corresponding operations on the level of linear maps. Theorem \ref{HomHspq} implies that the category of representations
of $S^+(p,0)$ is generated by the morphisms $\wt{T}_{\pi_1}$ and $\wt{T}_{\pi_2}$ associated with the partitions
\[ \pi_1= {\raisebox{10pt}{\xymatrix@R=0.3pc @C=0.3pc{ & & \\ & *=0{} &
\\ *=0{\bullet} \ar@{-}[ur] & & *=0{\circ} \ar@{-}[ul]}}}, \;\;\; \pi_2 = {\raisebox{10pt}{\xymatrix@R=0.3pc @C=0.3pc{  & *=0{\bullet} \ar@{-}[dd] & \\ & &
\\ *=0{\circ} \ar@{-}[uur] &*=0{\bullet} & *=0{\circ} \ar@{-}[uul]}}} .\]
The partitions $\pi_1$ and $\pi_2$ generate the whole $NC_{\bullet, \tu{even}}$. Indeed,  the analogous `non-bulleted'
partitions generate $NC_{\tu{even}}$, and it suffices to notice that $\pi_1$ and $\pi_2$ can be composed to obtain
the `exchange' partition $\pi_3= {\raisebox{5pt}{\xymatrix@R=0.7pc @C=0.7pc{*=0{\bullet} \ar@{-}[d] \\
 *=0{\circ}}}}\,$.  The generation statement is now clear.

The argument above shows that the category of the representations of $H^+_s(p,0)$ is at least as big as $NC_{\bullet,
\tu{even}}$; hence there is an inclusion  $\QG \supset S^+(p,0)$, where $\QG$ is the compact quantum group arising as the
Tannaka-Krein dual of $NC_{\bullet, \tu{even}}$ via the described implementation. To conclude the proof it suffices to show that
for each $k \in \bn$ the number $d_k:=\card(\Hom(1; U^{\ot 2k}))$ is equal to $\card(NC_{\bullet, \tu{even}}(0,2k))$. The numbers
$d_k$ are expressed by the formula
\[ d_k = h_{H^+_s(p,0)} \big(\Tr\, (U^{\ot 2k}) \big),\]
where $h_{H^+_s(p,0)}$ denotes the Haar state. Using the argument similar to that of Corollary 7.4 of \cite{stocheasy} and the
identification in the proof of Theorem \ref{thHspq} one can show that the variables $\mathbf{p}=p_A+ p_B+ p_C + p_D + \cdots$ and
$\mathbf{q}=q_A +q_B +q_C +q_D+ \cdots$ (see the matrix \eqref{fundhs}) are free Poisson variables of parameter $\frac{1}{4}$ --
note that this provides a concrete realisation of the variables found in \cite{Bessel}. Thus we are left with the computation of
the moments of 2 copies of a difference of such two free Poisson variables. This can be deduced from the results in
\cite{Bessel}: first note that according to Theorem 7.1 of \cite{Bessel} the $2k$-th moment of $\mathbf{p} - \mathbf{q}$ is  the
$k$-th moment of the free Bessel law with parameters $(2, \frac{1}{2})$,  and then combine the observation in the proof of
Theorem 5.2 of \cite{Bessel} with Theorem 4.3 of the same paper that the latter is equal to $\sum_{\pi \in NC_{\tu{even}}(0,2k)}
\left(\frac{1}{2}\right)^{\nu(\pi)}$, so that
\begin{equation} d_k = 2^k \sum_{\pi \in NC_{\tu{even}}(0,2k)}
\left(\frac{1}{2}\right)^{\nu(\pi)}=  \sum_{\pi \in NC_{\tu{even}}(0,2k)} 2^{k-\nu(\pi)} \label{Besselcount}\end{equation} (where
$\nu(\pi)$ denoted the number of blocks in the partition $\pi$). As in the proof of Theorem \ref{catSp} we observe  that the
number in \eqref{Besselcount} indeed equals $\card(NC_{\bullet}(0,2k))$ due to the multiplication factor expressing the choice of
colourings of bullets in each block of $\pi$.
\end{proof}

We are now ready to describe the category of partitions corresponding to the representation theory of $H^+_s(p,q)$.

\begin{tw} \label{catHsfree}
Let $U$ denote the fundamental representation of $H^+_s(p,q)$ introduced in Definition \ref{Hspq}.  Then for all $k,l \in \bn_0$
$$\Hom(U^{\otimes k};U^{\otimes l})=\tu{span}(\hat{T}_\pi|\pi\in (NC_{\bullet, \tu{even}} \star NC_{\tu{even}})(k,l)).$$
\end{tw}

\begin{proof}
We only provide the argument for the equality of the number of partitions in $(NC_{\bullet, \tu{even}} \star
NC_{\tu{even}})(0,k)$ and the $k$-th moment of the character of the defining representation of $H^+_s(p,q)$. The latter can be
expressed via the moment-cumulant formula \eqref{momcum}, with the free cumulants $(\kappa_k)_{k=1}^{\infty}$ given by the sum of
the free cumulants of the corresponding laws of the characters for $H^+_s(p,0)$ and $H^+_s(0,q)$.  As the second law is the free
Bessel distribution, another use of the moment-cumulant formula \eqref{momcum} and Theorem 4.3 of \cite{Bessel} implies that the
even free cumulants for it are equal to $1$, and the odd ones vanish. As the free Bessel laws form a free convolution semigroup,
Lemma \ref{lemcum} and the proof of Theorem \ref{catHsp} imply then that the odd free cumulants related to $H^+_s(0,q)$ vanish
and even ones are equal to $2^{k-1}$. Hence the moment-cumulant formula for the $k$-th moment of the law we are interested in
yields:
\[ d_k = \sum_{\pi \in NC_{\tu{even}}(0,k)}\,  \prod_{b: \tu{block in } \pi}\, (2^{\tu{card}(b)-1} +1).\]
The rest of the proof follows as in Theorem \ref{catfree}.

\end{proof}

\vspace*{0.2cm} \noindent \textbf{Acknowledgment.} Part of the work on this paper was began during the visit of the second-named author to
the Universit\'e Cergy-Pontoise in April 2010. The work of the first-named author was supported by the ANR grant `Galoisint'.

\end{document}